\documentclass[12pt,twoside,reqno]{amsart}
\usepackage[colorlinks=true,citecolor=blue]{hyperref}
\usepackage{mathptmx, amsmath, amssymb, amsfonts, amsthm, mathptmx, enumerate, color,mathrsfs}
\usepackage{graphicx}

\usepackage{epstopdf}

\newtheorem{theorem}{Theorem}[section]
\newtheorem{lemma}[theorem]{Lemma}

\newtheorem{corollary}[theorem]{Corollary}

\theoremstyle{definition}

\newtheorem{remark}[theorem]{Remark}
\numberwithin{equation}{section}

\begin{document}
\setcounter{page}{1}

\vspace*{2.0cm}
\title[Some Hardy and Rellich type inequalities for affine connections]
{Some Hardy and Rellich type inequalities for affine connections}
\author[P. Wang, H. Chang]{Pengyan Wang$^1$, Huiting Chang$^{1,*}$ }
\maketitle
\vspace*{-0.6cm}

\begin{center}
{\footnotesize

$^1$School of Mathematics and Statistics, Xinyang Normal University, Xinyang, 464000, P.R. China

}\end{center}

\vskip 4mm {\footnotesize \noindent {\bf Abstract.}
In this article we study various forms of the Hardy inequality for affine connections on a complete noncompact Riemannian manifold, including the two-weight Hardy inequality, the improved Hardy inequality, the Rellich inequality, the Hardy-Poincar\'{e} inequality and the Heisenberg-Pauli-Weyl inequality. Our results improve and include many previously known results as special cases.

}.


\renewcommand{\thefootnote}{}
\footnotetext{ $^*$Corresponding author.
\par
E-mail addresses: wangpy@xynu.edu.cn (P. Wang), changhuiting163@163.com (H. Chang).
}

\section{Introduction}

The classical Hardy inequality plays significant roles in many areas such as analysis, probability and partial differential equations:
$$\aligned
\int_{\mathbb{R}^n}\left|\nabla u(x)\right|^2dx\geq\left(\frac{n-2}{2}\right)^2\int_{\mathbb{R}^n}\frac{\left|u(x)\right|^2}{\left|x\right|^2}dx,
\endaligned$$
where $u\in C_0^\infty(\mathbb{R}^n\backslash\{0\})$ and $n\geq3$. A natural generalization of the Hardy inequality is the following Rellich inequality:
$$\aligned
\int_{\mathbb{R}^n}|\Delta u(x)|^2dx\geq\left(\frac{n(n-4)}4\right)^2\int_{\mathbb{R}^n}\frac{|u(x)|^2}{|x|^4}dx{\color{red}{,}}
\endaligned$$
where $u\in C_0^\infty(\mathbb{R}^n\backslash\{0\})$ and $n\geq5$.
The constant $\left(\frac{n-2}2\right)^2$ and $\left(\frac{n(n-4)}4\right)^2$ in the above inequalities are optimal and are never achieved by nontrivial functions.

In this article, we are concerned with various forms of the Hardy inequality for affine connections on a complete noncompact Riemannian manifold.

Let us provide the rigorous definitions we are going to employ for affine connections on Riemannian
manifolds (see \cite{LiXia2017}).

Let $(M, g)$ be an $n$-dimensional Riemannian manifold and $\nabla$ be the Levi-Civita
connection of $ g$. Let $V=e^{u}$, where $u$ is a smooth function defined on $M$. We call $(M, g, V)$ a $\textit{Riemannian triple}$.
For two real constants $\lambda$, $\mu$ and two vector fields $X$, $Y$ on $M$, we define a $2$-parameter family of \textit{affine connections} by
\begin{equation}\label{aInt1}
D^{\lambda,\mu}_{X}Y:=\nabla_{X}Y+\lambda \mathrm{d}u(X)Y+\lambda \mathrm{d}u(Y)X+\mu g(X,Y)\nabla u.
\end{equation}
As Li and Xia \cite{LiXia2017} have already noted, when $\lambda = \mu= 0$, $D^{\lambda,\mu}$ is a
Levi-Civita connection for $ g$; when $\lambda = -\mu$, $D^{\lambda,\mu}$ is a Levi-Civita connection for the conformal metric $e^{2\lambda f} g$.
Except for $\lambda = \mu= 0$ and $\lambda = -\mu$, $D^{\lambda,\mu}$ may not be a
Levi-Civita connection for any Riemannian metric.
The definitions of the affine Ricci curvature, the affine gradient, the affine Laplacian and the affine divergence will be given in Sect. 2 below.
For more details on affine connections, we refer to \cite{HuangM2023,Huangma2023} and the references
therein.

Now, let's recall some relevant work about the Hardy type inequalities for affine connections \eqref{aInt1}.
In the case $\lambda=\mu=0$, the affine connections \eqref{aInt1} reduces to the Levi-Civita connection. The affine gradient and the affine Laplacian are in consistence with the classical ones.
Carron \cite{Carron1997} first established the following Hardy inequality on Riemannian manifold:
$$\aligned
\int_M\rho^\alpha|\nabla u|^2d\nu_g\geq\frac{(C+\alpha-1)^2}4\int_M\rho^{\alpha-2}u^2d\nu_g,
\endaligned$$
where $\alpha\in\mathbb{R}, C+\alpha-1>0, u\in C_0^\infty\left(M\setminus\rho^{-1}\left\{0\right\}\right)$.
The weight function $\rho$ is nonnegative and it satisfies $|\nabla\rho|=1$
and $\Delta\rho\geq\frac C\rho$ in the sense of distribution. Throughout this paper, $d\nu_{g},\nabla$ and $\Delta$ denote the volume element, the gradient and the Laplace operators on $M$. Under the same geometric assumptions on the weight function $\rho$, Kombe and \"{O}zaydin \cite{Kombe2009} established the weighted $L^p$-Hardy inequality:
$$\aligned
\int_M\rho^\alpha\left|\nabla u\right|^pd\nu_g\geq\left(\frac{C+1+\alpha-p}p\right)^p\int_M\rho^{\alpha-p}u^p d\nu_g,
\endaligned$$
where $1\leq p<\infty,C+1+\alpha-p>0$ and $u\in C_0^\infty\left(M\setminus\rho^{-1}\left\{0\right\}\right)$. Kombe and \"{O}zaydin \cite{Kombe2013} also proved a new weighted Hardy-Poincar\'{e} inequalities. They showed that if $M$ is a complete non-compact Riemannian manifold of dimension $n>1$ and $\rho$ is a nonnegative function on M such that $|\nabla\rho|=1$ and $\Delta\rho\geq\frac C\rho$ in the sense of distribution, where $C>0$, the following inequality holds:
\begin{equation}\label{1adaInt2}\aligned
\int_M\rho^{\alpha+p}|\nabla\rho\cdot\nabla u|^pd\nu_g\geq\left(\frac{C+1+\alpha}p\right)^p\int_M\rho^\alpha|u|^pd\nu_g.
\endaligned\end{equation}
Kombe and Yener \cite{Kombe2016} provided a new form of the $L^p$-Hardy inequality involving two weight functions, that is,
$$\aligned
\int_{M}a(x)|\nabla u|^{p}~dv_{g}\geqslant\int_{M}b(x)|u|^{p}~dv_{g}+c(p)\int_{M}a(x)|\omega|^{p}\left|\nabla\left(\frac{u}{\omega}\right)\right|^{p}~dv_{g}
\endaligned$$
for suitable functions $a, b$ and $\omega$.
Xia \cite{Xia2014} proved the following Hardy type inequalities on a complete noncompact Riemannian manifold:
$$\aligned
\int_{M}\rho^{\alpha}|\nabla u|^{p}\geq\left(\frac{p-C-1-\alpha}{p}\right)^{p}
\int_{M}\rho^{\alpha-p}|u|^{p}-\left(\frac{p-C-1}{p}\right)^{p-1}
\int_{M}V\rho^{\alpha-p+1}|u|^{p}.
\endaligned$$
for any $p,\alpha\in\mathbb{R}$ with $p>1$ and $p-\alpha-C-1>0$. For more results on Hardy type inequalities on Riemannian manifolds, we refer to \cite{Flynn2023,Grillo2003,Ghoussoub2011,Ye2023,Jin2022,Kombe2010,Vazquez2000,Nguyen2020} and the references therein.

In the case $\lambda = -\mu$, the affine connections \eqref{aInt1} reduces to a Levi-Civita connection for the conformal metric $e^{2\alpha f} g$.
The affine gradient and the affine Laplacian are in consistence with the classical ones for the conformal metric $e^{2\alpha f} g$ (also known as the weighted Laplacian, for example, see \cite{HuangL2014,HuangZh2023,Li2005,Wangz2023}). Du and Mao \cite{Du2015} proved some Hardy and Rellich type inequalities for the weighted Laplacian. Li et al. \cite{Li2022} established some Hardy type inequalities with
remainder terms for the weighted Laplacian. While in \cite{Meng2021}, by means of various weighted Ricci curvatures, Meng et al. established several sharp Hardy type inequalities for the weighted Laplacian on closed weighted Riemannian manifolds.

As for the affine connections \eqref{aInt1}, to our knowledge, the only known
Hardy type inequality is due to Li and Xia \cite{LiXia2017}. They showed that, under some additional conditions on the affine Ricci curvature, the $L^{2}$-version of the Hardy type inequality holds on compact Riemannian manifolds (Theorem 5.1 in \cite{LiXia2017}).

Based on the above fact, a natural problem is for us how to generalize those Hardy type inequalities described above to the setting of complete Riemannian
manifold with affine connections \eqref{aInt1} and without any constraint on the affine Ricci curvature.

The motivation of this article is to answer the above problems.
The primary objectives of the present paper are twofold: Firstly, we want to establish
some $L^{2}$ improved Hardy type inequalities and $L^{p}$ Hardy type inequalities for affine connections on a complete Riemannian manifold.
Secondly, as application we would like to show some Rellich type inequalities, Hardy-Poincar\'{e} type inequalities and Heisenberg-Pauli-Weyl type inequalities for affine connections \eqref{aInt1}.

The rest of this paper is organized as follows. In Section 2, we recall the notations and the integration by parts under
affine connections. In Section 3, we establish $L^{2}$ two-weight Hardy type inequality
and $L^{p}$ Hardy type inequality for affine connections. Weighted Rellich type
inequalities for affine connections are discussed in Section 4. Finally, Section 5 is dedicated to the study of some inequalities of other types for affine connections.

\section{Preliminaries}

In this section, we recall some preliminary results in Li and Xia \cite{LiXia2017}.

\subsection{Torsion-free affine connections $D^{\lambda,\mu}$}

For two real constants $\lambda$, $\mu$ and two vector fields $X$, $Y$ on $M$, a $2$-parameter family of affine connections $D^{\lambda,\mu}$ defined by
\begin{equation}\label{01adsec02}
D^{\lambda,\mu}_{X}Y:=\nabla_{X}Y+\lambda \mathrm{d}u(X)Y+\lambda \mathrm{d}u(Y)X+\mu g(X,Y)\nabla u,
\end{equation}
where $u$ is a smooth function defined on $M$. One can check that $D^{\lambda,\mu}_{X}Y$ is \textit{torsion-free}.

\subsection{Affine gradient, affine Laplacian, affine divergence and affine Ricci curvature.}
Let $\phi$ be a smooth function on $M^{n}$.
For an $n$-dimensional smooth Riemannian triple $(M^{n}, g, V=e^{u})$.

(1) The affine gradient on $M$ is defined by
\begin{equation}\label{ad01Sec1}
\aligned
\nabla^{D^{\lambda,\mu}}\phi:=V^{\mu-\lambda}\nabla\phi.
\endaligned
\end{equation}

(2) The affine Laplacian $\Delta^{D^{\lambda,\mu}}\phi$ on $M$ is defined by
\begin{equation}\label{ad02Sec1}
\aligned
\Delta^{D^{\lambda,\mu}}\phi:=V^{\mu-\lambda}[\Delta\phi+(2\mu+n\lambda)\langle\nabla u, \nabla\phi\rangle].
\endaligned
\end{equation}

(3) Let $X=\nabla^{D^{\lambda,\mu}}\phi$. The affine divergence $\mathrm{div}^{D^{\lambda,\mu}}X$ on $M$ is defined by
\begin{equation}\label{ad01ad02Sec1}
\aligned
\mathrm{div}^{D^{\lambda,\mu}}X:= D^{\lambda,\mu}_{i}X^{i}= V^{\mu-\lambda}(\Delta\phi+(2\mu+n\lambda)\langle\nabla u,\nabla\phi\rangle)= \Delta^{D^{\lambda,\mu}}\phi.
\endaligned
\end{equation}

(4) The affine Ricci curvature on $M$ is defined by:
\begin{equation}\label{aInt2}\aligned
{\rm Ric}^{D^{\lambda,\mu}}:=&{\rm Ric}-[(n-1)\lambda+\mu]\nabla^{2}u+\left[(n-1)\lambda^{2}-\mu^{2}\right]\mathrm{d}u\otimes\mathrm{d}u\\
&+\left[\mu\Delta u+(\mu^{2}+(n-1)\lambda\mu)|\nabla u|^{2}\right]g.
\endaligned\end{equation}

(5) The integration by parts with respect to affine connections as follows:
\begin{equation}\label{ad03Sec1}
\aligned
\int_{M}V^{\tau}\omega\Delta^{D^{\lambda,\mu}}\phi ~dv_{g}= -\int_{M}\langle\nabla \omega,\nabla^{D^{\lambda,\mu}}\phi\rangle V^{\tau}~dv_{g}= \int_{M}V^{\tau}\phi\Delta^{D^{\lambda,\mu}}\omega ~dv_{g}
\endaligned
\end{equation}
where $\omega$ and $\phi$ are smooth functions on $M$.

\begin{lemma}\label{lem2-2}(\cite{LiXia2017})
Let $W$ be any smooth vector field on $M$. Then
\begin{equation}\label{z3lem1}
V^{\tau}D^{\lambda,\mu}_{i}W^{i}=\nabla_{i}(V^{\tau}W^{i})
\end{equation}
is a divergent form with respect to the Riemannian volume form $dv_{g}$, where $\tau=(n+1)\lambda+\mu$ and we adopt the Einstein convention.
\end{lemma}

\section{Hardy type inequalities and their proofs}
In this section, we will prove some integral inequalities of the Hardy type for affine connections $D^{\lambda,\mu}$.
Our first result is an $L^{2}$-version of the affine Hardy inequality without remainder term given as follows.

\begin{theorem}\label{thm1}
Let $(M^{n}, g,V=e^{u})$ be an n-dimensional Riemannian triple and
$\lambda,\mu\in \mathbb{R}$. Let $D=D^{\lambda,\mu}$ be the affine connection defined as in \eqref{01adsec02} and $\tau=(n+1)\lambda+\mu$.
Let $\rho$ be a nonnegative function on $M$ such that $|\nabla^{D}\rho|=1$ in the sense of distributions. Then for
any compactly supported smooth function $\phi\in C^{\infty}_{0} (M\backslash\rho^{-1}\{0\})$, $1<p<+\infty$, we have

$\mathrm{(i)}$\quad When $\Delta^{D}\rho\leq \frac{C}{\rho}V^{\lambda-\mu}$ in the sense of distributions, where $C>0$ is a constant and $1-\alpha-C>0$, the
following inequality
 \begin{equation}\label{01Sec2thm1}\aligned
\int_{M}\rho^{\alpha}|\nabla^{D}\phi|^{2}V^{\tau+\lambda-\mu}dv_{g}\geq
\left(\frac{1-\alpha-C}{2}\right)^{2}\int_{M}\rho^{\alpha-2}\phi^{2}V^{\tau+\lambda-\mu}dv_{g}
\endaligned\end{equation}
holds.

$\mathrm{(ii)}$\quad When $\Delta^{D}\rho\geq \frac{C}{\rho}V^{\lambda-\mu}$ in the sense of distributions, where $C>0$ is a constant and $C+\alpha-1>0$, the
following inequality
 \begin{equation}\label{02Sec2thm1}\aligned
\int_{M}\rho^{\alpha}|\nabla^{D}\phi|^{2}V^{\tau+\lambda-\mu}dv_{g}\geq
\left(\frac{\alpha+C-1}{2}\right)^{2}\int_{M}\rho^{\alpha-2}\phi^{2}V^{\tau+\lambda-\mu}dv_{g}
\endaligned\end{equation}
holds.
\end{theorem}
\begin{proof}

$\mathrm{(i)}$\quad Let $\gamma=\frac{1-\alpha-C}{2}$. Clearly, $\gamma>0$. By direct computation, we obtain
\begin{equation}\label{03Sec2thm1}\aligned
\Delta^{D}\left(\rho^{\alpha+2\gamma}\right)=&V^{\mu-\lambda}\left(\Delta\rho^{\alpha+2\gamma}+(2\mu+n\lambda)\langle\nabla u,\nabla \rho^{\alpha+2\gamma}\rangle\right)\\
=&V^{\mu-\lambda}\big[(\alpha+2\gamma)\left((\alpha+2\gamma-1)|\nabla\rho|^{2}\rho^{\alpha+2\gamma-2}+\rho^{\alpha+2\gamma-1}\Delta \rho\right)\\
&+(2\mu+n\lambda)(\alpha+2\gamma)\rho^{\alpha+2\gamma-1}\langle\nabla u,\nabla \rho\rangle\big]\\
=&(\alpha+2\gamma)(\alpha+2\gamma-1)\rho^{\alpha+2\gamma-2}|\nabla^{D}\rho|^{2}V^{\lambda-\mu}+(\alpha+2\gamma)\rho^{\alpha+2\gamma-1}\Delta^{D}\rho\\
=&(\alpha+2\gamma)\left((\alpha+2\gamma-1)|\nabla^{D}\rho|^{2}\rho^{\alpha+2\gamma-2}V^{\lambda-\mu}+\rho^{\alpha+2\gamma-1}\Delta^{D}\rho\right),
\endaligned\end{equation}
where, we use \eqref{ad02Sec1} in the first equation and also use \eqref{ad01Sec1} and \eqref{ad02Sec1} in the last second equation.
Together with $|\nabla^{D}\rho|=1$ and $\Delta^{D}\rho\leq \frac{C}{\rho}V^{\lambda-\mu}$, we can get that for $C\neq1 $, the following inequality holds:
\begin{equation}\label{04Sec2thm1}\aligned
&\frac{1}{\alpha+2\gamma}\Delta^{D}(\rho^{\alpha+2\gamma})\\
=&(\alpha+2\gamma-1)|\nabla^{D}\rho|^{2}\rho^{\alpha+2\gamma-2}V^{\lambda-\mu}+\rho^{\alpha+2\gamma-1}\Delta^{D}\rho\\
\leq&(\alpha+2\gamma-1+C)\rho^{\alpha+2\gamma-2}V^{\lambda-\mu}\\
=&0.
\endaligned\end{equation}
Let $\phi=\rho^{\gamma}\psi$, where $\psi\in C^{\infty}_{0} (M\backslash\rho^{-1}\{0\})$. Simple calculation can be obtained,
\begin{equation}\label{05Sec2thm1}\aligned
&\rho^{\alpha}|\nabla\phi|^{2}\\
=&\rho^{\alpha}|\gamma\rho^{\gamma-1}\psi\nabla\rho+\rho^{\gamma}\nabla\psi|^{2}\\
=&\rho^{\alpha}\left(\gamma^{2}\rho^{2(\gamma-1)}\psi^{2}|\nabla\rho|^{2}+\rho^{2\gamma}|\nabla\psi|^{2}+2\gamma\rho^{2\gamma-1}\psi\langle\nabla \rho,\nabla\psi\rangle\right)\\
=&\gamma^{2}\rho^{\alpha+2(\gamma-1)}\psi^{2}|\nabla\rho|^{2}+\rho^{\alpha+2\gamma}|\nabla\psi|^{2}+2\gamma\rho^{\alpha+2\gamma-1}\psi\langle\nabla \rho,\nabla\psi\rangle\\
=&\gamma^{2}\rho^{\alpha+2(\gamma-1)}\psi^{2}|\nabla\rho|^{2}+\rho^{\alpha+2\gamma}|\nabla\psi|^{2}+\frac{\gamma}{\alpha+2\gamma}\langle \nabla \rho^{\alpha+2\gamma},\nabla\psi^{2}\rangle\\
=&\gamma^{2}\rho^{\alpha+2(\gamma-1)}\psi^{2}|\nabla^{D}\rho|^{2}V^{2(\lambda-\mu)}+\rho^{\alpha+2\gamma}|\nabla^{D}\psi|^{2}V^{2(\lambda-\mu)}+\frac{\gamma}{\alpha+2\gamma}\langle V^{\lambda-\mu}\nabla^{D} \rho^{\alpha+2\gamma},\nabla\psi^{2}\rangle \\
\geq&\gamma^{2}\rho^{\alpha+2(\gamma-1)}\psi^{2}|\nabla^{D}\rho|^{2}V^{2(\lambda-\mu)}+\frac{\gamma}{\alpha+2\gamma}\langle V^{\lambda-\mu}\nabla^{D} \rho^{\alpha+2\gamma},\nabla\psi^{2}\rangle.
\endaligned\end{equation}
Multiplying both sides of the above inequality by $V^{\tau+\mu-\lambda}$ and using $|\nabla^{D}\rho|=1$, we get
\begin{equation}\label{06Sec2thm1}\aligned
\rho^{\alpha}|\nabla\phi|^{2}V^{\tau+\mu-\lambda}\geq\gamma^{2}\rho^{\alpha+2\gamma-2}\psi^{2}V^{\tau+\lambda-\mu}
+\frac{\gamma}{\alpha+2\gamma}\langle V^{\tau}\nabla^{D} \rho^{\alpha+2\gamma},\nabla\psi^{2}\rangle.
\endaligned\end{equation}
Integrating \eqref{06Sec2thm1} over $M$ yields
\begin{equation}\label{07Sec2thm1}\aligned
&\int_{M}\rho^{\alpha}|\nabla\phi|^{2}V^{\tau+\mu-\lambda}~dv_{g}\\
\geq&\gamma^{2}\int_{M}\rho^{\alpha+2\gamma-2}\psi^{2}V^{\tau+\lambda-\mu}~dv_{g}
+\frac{\gamma}{\alpha+2\gamma}\int_{M}\langle V^{\tau}\nabla^{D} \rho^{\alpha+2\gamma},\nabla\psi^{2}\rangle ~dv_{g}.
\endaligned\end{equation}
We use the integration by parts \eqref{ad03Sec1} to derive that
\begin{equation}\label{08Sec2thm1}\aligned
&\int_{M}\langle V^{\tau}\nabla^{D} \rho^{\alpha+2\gamma},\nabla\psi^{2}\rangle ~dv_{g}\\
=&\int_{M} V^{\tau}\nabla^{D}_{i} \rho^{\alpha+2\gamma}\nabla_{i}\psi^{2} ~dv_{g}\\
=&-\int_{M}\psi^{2}\nabla_{i}\left(V^{\tau}\nabla^{D}_{i}\rho^{\alpha+2\gamma}\right)~dv_{g}\\
=&-\int_{M}\psi^{2}V^{\tau}D_{i}\left(\nabla^{D}_{i}\rho^{\alpha+2\gamma}\right)~dv_{g}\\
=&-\int_{M}\psi^{2}V^{\tau}\Delta^{D} \rho^{\alpha+2\gamma}~dv_{g},
\endaligned\end{equation}
where, in the last second inequality, we use \eqref{z3lem1}.
Inserting \eqref{08Sec2thm1} into \eqref{07Sec2thm1}, we attain
$$\aligned
&\int_{M}\rho^{\alpha}|\nabla\phi|^{2}V^{\tau+\mu-\lambda}~dv_{g}\\
\geq&\gamma^{2}\int_{M}\rho^{\alpha+2\gamma-2}\psi^{2}V^{\tau+\lambda-\mu}~dv_{g}
-\frac{\gamma}{\alpha+2\gamma}\int_{M}\psi^{2}V^{\tau}\Delta^{D} \rho^{\alpha+2\gamma}~dv_{g}.
\endaligned $$
Applying \eqref{04Sec2thm1} and $\phi=\rho^{\gamma}\psi$, we obtain
\begin{equation}\label{008Sec2thm1}\aligned
&\int_{M}\rho^{\alpha}|\nabla\phi|^{2}V^{\tau+\mu-\lambda}~dv_{g}\\
\geq&\gamma^{2}\int_{M}\rho^{\alpha+2\gamma-2}\psi^{2}V^{\tau+\lambda-\mu}~dv_{g}\\
=&\left(\frac{1-\alpha-C}{2}\right)^{2}\int_{M}\rho^{\alpha-2}\phi^{2}V^{\tau+\lambda-\mu}~dv_{g}.
\endaligned\end{equation}

On the other hand, clearly, when $C=1$, by applying a similar computation as above, we can get
\begin{equation}\label{009Sec2thm1}\aligned
\Delta^{D}\left(\ln\rho\right)=&V^{\mu-\lambda}\left(\Delta\ln\rho+(2\mu+n\lambda)\langle\nabla u,\nabla \ln\rho\rangle\right)\\
=&V^{\mu-\lambda}\left(-\frac{1}{\rho^{2}}|\nabla\rho|^{2}+\frac{1}{\rho}\Delta \rho
+(2\mu+n\lambda)\frac{1}{\rho}\langle\nabla u,\nabla \rho\rangle \right)\\
=&-\frac{1}{\rho^{2}}|\nabla^{D}\rho|^{2}V^{\lambda-\mu}+\frac{1}{\rho}\Delta^{D}\rho
\endaligned\end{equation}
and
\begin{equation}\label{09Sec2thm1}\aligned
\rho^{\alpha}|\nabla\phi|^{2}
=&\gamma^{2}\rho^{\alpha+2(\gamma-1)}\psi^{2}|\nabla\rho|^{2}+\rho^{\alpha+2\gamma}|\nabla\psi|^{2}+2\gamma\rho^{\alpha+2\gamma-1}\psi\langle\nabla \rho,\nabla\psi\rangle\\
=&\gamma^{2}\rho^{\alpha+2(\gamma-1)}\psi^{2}|\nabla\rho|^{2}+\rho^{\alpha+2\gamma}|\nabla\psi|^{2}+2\gamma\rho^{-1}\psi\langle\nabla \rho,\nabla\psi\rangle\\
=&\gamma^{2}\rho^{\alpha+2(\gamma-1)}\psi^{2}|\nabla^{D}\rho|^{2}V^{2(\lambda-\mu)}+\rho^{\alpha+2\gamma}|\nabla^{D}\psi|^{2}V^{2(\lambda-\mu)}+2\gamma\rho^{-1}\psi\langle\nabla \rho,\nabla\psi\rangle\\
\geq&\gamma^{2}\rho^{\alpha+2(\gamma-1)}\psi^{2}|\nabla^{D}\rho|^{2}V^{2(\lambda-\mu)}+\gamma\langle\nabla \ln\rho,\nabla\psi^{2}\rangle.
\endaligned\end{equation}
Together with $|\nabla^{D}\rho|=1$ and $\Delta^{D}\rho\leq \frac{C}{\rho}V^{\lambda-\mu}$ in \eqref{009Sec2thm1}, we can get that for $C=1 $, the following inequality holds:
\begin{equation}\label{0901adSec2thm1}\aligned
\Delta^{D}\left(\ln\rho\right)
\leq&(C-1)\frac{1}{\rho^{2}}V^{\lambda-\mu}\\
=&0.
\endaligned\end{equation}
Multiplying both sides in \eqref{09Sec2thm1} by $V^{\tau+\mu-\lambda}$ and integrating over $M$ yields
\begin{equation}\label{10Sec2thm1}\aligned
\int_{M}\rho^{\alpha}|\nabla\phi|^{2}V^{\tau+\mu-\lambda}~dv_{g}
\geq&\gamma^{2}\int_{M}\rho^{\alpha+2\gamma-2}\psi^{2}|\nabla^{D}\rho|^{2}V^{\tau+\lambda-\mu}~dv_{g}\\
&+\gamma\int_{M}\langle \nabla \ln\rho,\nabla\psi^{2}\rangle V^{\tau+\mu-\lambda} ~dv_{g}.
\endaligned\end{equation}
Applying the integration by parts, we get
\begin{equation}\label{11Sec2thm1}\aligned
&\int_{M}\langle \nabla \ln\rho,\nabla\psi^{2}\rangle V^{\tau+\mu-\lambda} ~dv_{g}\\
=&\int_{M}\langle V^{\tau}\nabla^{D} \ln\rho,\nabla\psi^{2}\rangle ~dv_{g}\\
=&\int_{M}V^{\tau}\nabla^{D}_{i} (\ln\rho)\nabla_{i}\psi^{2} ~dv_{g}\\
=&-\int_{M}\psi^{2}\nabla_{i}\left(V^{\tau} \nabla^{D}_{i}\ln\rho\right) ~dv_{g}\\
=&-\int_{M}\psi^{2}V^{\tau}D_{i}\left( \nabla^{D}_{i}\ln\rho\right) ~dv_{g}\\
=&-\int_{M}V^{\tau}\psi^{2}\Delta^{D}(\ln\rho) ~dv_{g}.
\endaligned\end{equation}
Then combining \eqref{10Sec2thm1} and \eqref{11Sec2thm1}, we obtain
\begin{equation}\label{12Sec2thm1}\aligned
&\int_{M}\rho^{\alpha}|\nabla\phi|^{2}V^{\tau+\mu-\lambda}~dv_{g}\\
\geq&\gamma^{2}\int_{M}\rho^{\alpha+2\gamma-2}\psi^{2}V^{\tau+\lambda-\mu}~dv_{g}
-\gamma\int_{M}V^{\tau}\psi^{2}\Delta^{D}(\ln\rho) ~dv_{g},
\endaligned\end{equation}
where $|\nabla^{D}\rho|=1$.
Substituting \eqref{0901adSec2thm1} into \eqref{12Sec2thm1}, since $\gamma=-\frac{\alpha}{2}>0$ and $\phi=\rho^{\gamma}\psi$, we can get
\begin{equation}\label{13Sec2thm1}\aligned
&\int_{M}\rho^{\alpha}|\nabla\phi|^{2}V^{\tau+\mu-\lambda}~dv_{g}\\
=&\int_{M}\rho^{\alpha}|\nabla^{D}\phi|^{2}V^{\tau+\lambda-\mu}dv_{g}\\
\geq&\left(-\frac{\alpha}{2}\right)^{2}\int_{M}\rho^{\alpha-2}\phi^{2}V^{\tau+\lambda-\mu}~dv_{g}\\
=&\left(\frac{\alpha}{2}\right)^{2}\int_{M}\rho^{\alpha-2}\phi^{2}V^{\tau+\lambda-\mu}~dv_{g}.
\endaligned\end{equation}
By \eqref{13Sec2thm1}, we know that \eqref{01Sec2thm1} is true.

$\mathrm{(ii)}$\quad Let $\gamma=\frac{1-\alpha-C}{2}$. Clearly, $\gamma<0$. By direct computation, we have
\begin{equation}\label{03Sec2thm1-2}\aligned
\Delta^{D}\left(\rho^{\alpha+2\gamma}\right)=&V^{\mu-\lambda}\left(\Delta\rho^{\alpha+2\gamma}+(2\mu+n\lambda)\langle\nabla u,\nabla \rho^{\alpha+2\gamma}\rangle\right)\\
=&V^{\mu-\lambda}\big[(\alpha+2\gamma)\left((\alpha+2\gamma-1)|\nabla\rho|^{2}\rho^{\alpha+2\gamma-2}+\rho^{\alpha+2\gamma-1}\Delta \rho\right)\\
&+(2\mu+n\lambda)(\alpha+2\gamma)\rho^{\alpha+2\gamma-1}\langle\nabla u,\nabla \rho\rangle\big]\\
=&(\alpha+2\gamma)(\alpha+2\gamma-1)\rho^{\alpha+2\gamma-2}|\nabla^{D}\rho|^{2}V^{\lambda-\mu}+(\alpha+2\gamma)\rho^{\alpha+2\gamma-1}\Delta^{D}\rho\\
=&(\alpha+2\gamma)\left((\alpha+2\gamma-1)|\nabla^{D}\rho|^{2}\rho^{\alpha+2\gamma-2}V^{\lambda-\mu}+\rho^{\alpha+2\gamma-1}\Delta^{D}\rho\right),
\endaligned\end{equation}
where, we use \eqref{ad02Sec1} in the first equation and also use \eqref{ad01Sec1} and \eqref{ad02Sec1} in the last second equation. Together with $|\nabla^{D}\rho|=1$ and $\Delta^{D}\rho\geq \frac{C}{\rho}V^{\lambda-\mu}$, we can get that for $C\neq1 $, the following inequality holds:
\begin{equation}\label{04Sec2thm1-2}\aligned
&\frac{1}{\alpha+2\gamma}\Delta^{D}(\rho^{\alpha+2\gamma})\\
=&(\alpha+2\gamma-1)|\nabla^{D}\rho|^{2}\rho^{\alpha+2\gamma-2}V^{\lambda-\mu}+\rho^{\alpha+2\gamma-1}\Delta^{D}\rho\\
\geq&(\alpha+2\gamma-1+C)\rho^{\alpha+2\gamma-2}V^{\lambda-\mu}\\
=&0.
\endaligned\end{equation}
Let $\phi=\rho^{\gamma}\psi$, where $\psi\in C^{\infty}_{0} (M\backslash\rho^{-1}\{0\})$. Simple calculation can be obtained,
\begin{equation}\label{05Sec2thm1-2}\aligned
&\rho^{\alpha}|\nabla\phi|^{2}\\
=&\rho^{\alpha}|\gamma\rho^{\gamma-1}\psi\nabla\rho+\rho^{\gamma}\nabla\psi|^{2}\\
=&\rho^{\alpha}\left(\gamma^{2}\rho^{2(\gamma-1)}\psi^{2}|\nabla\rho|^{2}+\rho^{2\gamma}|\nabla\psi|^{2}+2\gamma\rho^{2\gamma-1}\psi\langle\nabla \rho,\nabla\psi\rangle\right)\\
=&\gamma^{2}\rho^{\alpha+2(\gamma-1)}\psi^{2}|\nabla\rho|^{2}+\rho^{\alpha+2\gamma}|\nabla\psi|^{2}+2\gamma\rho^{\alpha+2\gamma-1}\psi\langle\nabla \rho,\nabla\psi\rangle\\
=&\gamma^{2}\rho^{\alpha+2(\gamma-1)}\psi^{2}|\nabla\rho|^{2}+\rho^{\alpha+2\gamma}|\nabla\psi|^{2}+\frac{\gamma}{\alpha+2\gamma}\langle \nabla \rho^{\alpha+2\gamma},\nabla\psi^{2}\rangle\\
=&\gamma^{2}\rho^{\alpha+2(\gamma-1)}\psi^{2}|\nabla^{D}\rho|^{2}V^{2(\lambda-\mu)}+\rho^{\alpha+2\gamma}|\nabla^{D}\psi|^{2}V^{2(\lambda-\mu)}+\frac{\gamma}{\alpha+2\gamma}\langle V^{\lambda-\mu}\nabla^{D} \rho^{\alpha+2\gamma},\nabla\psi^{2}\rangle \\
\geq&\gamma^{2}\rho^{\alpha+2(\gamma-1)}\psi^{2}|\nabla^{D}\rho|^{2}V^{2(\lambda-\mu)}+\frac{\gamma}{\alpha+2\gamma}\langle V^{\lambda-\mu}\nabla^{D} \rho^{\alpha+2\gamma},\nabla\psi^{2}\rangle.
\endaligned\end{equation}
Multiplying both sides of the above inequality by $V^{\tau+\mu-\lambda}$ and using $|\nabla^{D}\rho|=1$, we get
\begin{equation}\label{06Sec2thm1-2}\aligned
\rho^{\alpha}|\nabla\phi|^{2}V^{\tau+\mu-\lambda}\geq\gamma^{2}\int_{M}\rho^{\alpha+2\gamma-2}\psi^{2}V^{\tau+\lambda-\mu}
+\frac{\gamma}{\alpha+2\gamma}\langle V^{\tau}\nabla^{D} \rho^{\alpha+2\gamma},\nabla\psi^{2}\rangle.
\endaligned\end{equation}
Integrating \eqref{06Sec2thm1-2} over $M$ yields
\begin{equation}\label{07Sec2thm1-2}\aligned
&\int_{M}\rho^{\alpha}|\nabla\phi|^{2}V^{\tau+\mu-\lambda}~dv_{g}\\
\geq&\gamma^{2}\int_{M}\rho^{\alpha+2\gamma-2}\psi^{2}V^{\tau+\lambda-\mu}~dv_{g}
+\frac{\gamma}{\alpha+2\gamma}\int_{M}\langle V^{\tau}\nabla^{D} \rho^{\alpha+2\gamma},\nabla\psi^{2}\rangle ~dv_{g}.
\endaligned\end{equation}
We use the integration by parts \eqref{ad03Sec1} to derive that
\begin{equation}\label{08Sec2thm1-2}\aligned
&\int_{M}\langle V^{\tau}\nabla^{D} \rho^{\alpha+2\gamma},\nabla\psi^{2}\rangle ~dv_{g}\\
=&\int_{M} V^{\tau}\nabla^{D}_{i} \rho^{\alpha+2\gamma}\nabla_{i}\psi^{2} ~dv_{g}\\
=&-\int_{M}\psi^{2}\nabla_{i}\left(V^{\tau}\nabla^{D}_{i}\rho^{\alpha+2\gamma}\right)~dv_{g}\\
=&-\int_{M}\psi^{2}V^{\tau}D_{i}\left(\nabla^{D}_{i}\rho^{\alpha+2\gamma}\right)~dv_{g}\\
=&-\int_{M}\psi^{2}V^{\tau}\Delta^{D} \rho^{\alpha+2\gamma}~dv_{g},
\endaligned\end{equation}
where, in the last second inequality, we use \eqref{z3lem1}.
Inserting \eqref{08Sec2thm1-2} into \eqref{07Sec2thm1-2}, we attain
$$\aligned
&\int_{M}\rho^{\alpha}|\nabla\phi|^{2}V^{\tau+\mu-\lambda}~dv_{g}\\
\geq&\gamma^{2}\int_{M}\rho^{\alpha+2\gamma-2}\psi^{2}V^{\tau+\lambda-\mu}~dv_{g}
-\frac{\gamma}{\alpha+2\gamma}\int_{M}\psi^{2}\Delta^{D} \rho^{\alpha+2\gamma}V^{\tau}~dv_{g}.
\endaligned $$
Applying \eqref{04Sec2thm1-2} and $\phi=\rho^{\gamma}\psi$, we obtain
\begin{equation}\label{08Sec2thm1-2}\aligned
&\int_{M}\rho^{\alpha}|\nabla\phi|^{2}V^{\tau+\mu-\lambda}~dv_{g}\\
\geq&\gamma^{2}\int_{M}\rho^{\alpha+2\gamma-2}\psi^{2}V^{\tau+\lambda-\mu}~dv_{g}\\
=&\left(\frac{1-\alpha-C}{2}\right)^{2}\int_{M}\rho^{\alpha-2}\phi^{2}V^{\tau+\lambda-\mu}~dv_{g}\\
=&\left(\frac{\alpha+C-1}{2}\right)^{2}\int_{M}\rho^{\alpha-2}\phi^{2}V^{\tau+\lambda-\mu}dv_{g}.
\endaligned\end{equation}

On the other hand, clearly,when $C=1$, by applying a similar computation as above, we can get
\begin{equation}\label{03Sec2thm1-2}\aligned
\Delta^{D}\left(\ln\rho\right)=&V^{\mu-\lambda}\left(\Delta\ln\rho+(2\mu+n\lambda)\langle\nabla u,\nabla \ln\rho\rangle\right)\\
=&V^{\mu-\lambda}\left(-\frac{1}{\rho^{2}}|\nabla\rho|^{2}+\frac{1}{\rho}\Delta \rho
+(2\mu+n\lambda)\frac{1}{\rho}\langle\nabla u,\nabla \rho\rangle \right)\\
=&-\frac{1}{\rho^{2}}|\nabla^{D}\rho|^{2}V^{\lambda-\mu}+\frac{1}{\rho}\Delta^{D}\rho
\endaligned\end{equation}
and
\begin{equation}\label{09Sec2thm1-2}\aligned
\rho^{\alpha}|\nabla\phi|^{2}
=&\gamma^{2}\rho^{\alpha+2(\gamma-1)}\psi^{2}|\nabla\rho|^{2}+\rho^{\alpha+2\gamma}|\nabla\psi|^{2}+2\gamma\rho^{\alpha+2\gamma-1}\psi\langle\nabla \rho,\nabla\psi\rangle\\
=&\gamma^{2}\rho^{\alpha+2(\gamma-1)}\psi^{2}|\nabla\rho|^{2}+\rho^{\alpha+2\gamma}|\nabla\psi|^{2}+2\gamma\rho^{-1}\psi\langle\nabla \rho,\nabla\psi\rangle\\
=&\gamma^{2}\rho^{\alpha+2(\gamma-1)}\psi^{2}|\nabla^{D}\rho|^{2}V^{2(\lambda-\mu)}+\rho^{\alpha+2\gamma}|\nabla^{D}\psi|^{2}V^{2(\lambda-\mu)}+2\gamma\rho^{-1}\psi\langle\nabla \rho,\nabla\psi\rangle\\
\geq&\gamma^{2}\rho^{\alpha+2(\gamma-1)}\psi^{2}|\nabla^{D}\rho|^{2}V^{2(\lambda-\mu)}+\gamma\langle\nabla \ln\rho,\nabla\psi^{2}\rangle.
\endaligned\end{equation}
Together with $|\nabla^{D}\rho|=1$ and $\Delta^{D}\rho\geq \frac{C}{\rho}V^{\lambda-\mu}$ in \eqref{009Sec2thm1}, we can get that for $C=1 $, the following inequality holds:
\begin{equation}\label{0901adSec2thm1-2}\aligned
\Delta^{D}\left(\ln\rho\right)
\geq&(C-1)\frac{1}{\rho^{2}}V^{\lambda-\mu}\\
=&0.
\endaligned\end{equation}
Multiplying both sides in \eqref{09Sec2thm1-2} by $V^{\tau+\mu-\lambda}$ and integrating over $M$ yields
\begin{equation}\label{10Sec2thm1-2}\aligned
&\int_{M}\rho^{\alpha}|\nabla\phi|^{2}V^{\tau+\mu-\lambda}~dv_{g}\\
\geq&\gamma^{2}\int_{M}\rho^{\alpha+2\gamma-2}\psi^{2}|\nabla^{D}\rho|^{2}V^{\tau+\lambda-\mu}~dv_{g}
+\gamma\int_{M}\langle \nabla \ln\rho,\nabla\psi^{2}\rangle V^{\tau+\mu-\lambda} ~dv_{g}.
\endaligned\end{equation}
Applying the integration by parts, we get
\begin{equation}\label{11Sec2thm1-2}\aligned
&\int_{M}\langle \nabla \ln\rho,\nabla\psi^{2}\rangle V^{\tau+\mu-\lambda} ~dv_{g}\\
=&\int_{M}\langle V^{\tau}\nabla^{D} \ln\rho,\nabla\psi^{2}\rangle ~dv_{g}\\
=&\int_{M}V^{\tau}\nabla^{D}_{i} (\ln\rho)\nabla_{i}\psi^{2} ~dv_{g}\\
=&-\int_{M}\psi^{2}\nabla_{i}\left(V^{\tau} \nabla^{D}_{i}\ln\rho\right) ~dv_{g}\\
=&-\int_{M}\psi^{2}V^{\tau}D_{i}\left( \nabla^{D}_{i}\ln\rho\right) ~dv_{g}\\
=&-\int_{M}V^{\tau}\psi^{2}\Delta^{D}(\ln\rho) ~dv_{g}.
\endaligned\end{equation}
Then combining \eqref{10Sec2thm1-2} and \eqref{11Sec2thm1-2}, we obtain
\begin{equation}\label{12Sec2thm1-2}\aligned
&\int_{M}\rho^{\alpha}|\nabla\phi|^{2}V^{\tau+\mu-\lambda}~dv_{g}\\
\geq&\gamma^{2}\int_{M}\rho^{\alpha+2\gamma-2}\psi^{2}V^{\tau+\lambda-\mu}~dv_{g}
-\gamma\int_{M}V^{\tau}\psi^{2}\Delta^{D}(\ln\rho) ~dv_{g},
\endaligned\end{equation}
where $|\nabla^{D}\rho|=1$. Substituting \eqref{0901adSec2thm1-2} into \eqref{12Sec2thm1-2}, since $\gamma=-\frac{\alpha}{2}<0$ and $\phi=\rho^{\gamma}\psi$, we can get
\begin{equation}\label{13Sec2thm1-2}\aligned
&\int_{M}\rho^{\alpha}|\nabla\phi|^{2}V^{\tau+\mu-\lambda}~dv_{g}\\
=&\int_{M}\rho^{\alpha}|\nabla^{D}\phi|^{2}V^{\tau+\lambda-\mu}dv_{g}\\
\geq&\left|-\frac{\alpha}{2}\right|^{2}\int_{M}\rho^{\alpha-2}\phi^{2}V^{\tau+\lambda-\mu}~dv_{g}\\
\geq&\left|\frac{\alpha}{2}\right|^{2}\int_{M}\rho^{\alpha-2}\phi^{2}V^{\tau+\lambda-\mu}~dv_{g}.
\endaligned\end{equation}
By \eqref{13Sec2thm1-2}, we know that \eqref{02Sec2thm1} is true. This completes the proof of Theorem \ref{thm1}.
\end{proof}

Our second result is $L^{p}$-version of the Hardy type inequality without remainder term given as follows.

\begin{theorem}\label{thm2}
Let $(M^{n}, g,V=e^{u})$ be an n-dimensional Riemannian triple and
$\lambda,\mu\in \mathbb{R}$. Let $D=D^{\lambda,\mu}$ be the affine connection defined as in \eqref{01adsec02} and $\tau=(n+1)\lambda+\mu$.
Let $\rho$ be a nonnegative function on $M$ such that $|\nabla^{D}\rho|=1$ in the sense of distributions.
Then, for any $p, q\in\mathbb{R}$, $1<p<+\infty$, $0\leq q \leq p$, and any compactly supported smooth function $\phi\in C^{\infty}_{0} (M\backslash\rho^{-1}\{0\})$, we have

$\mathrm{(i)}$\quad When $\Delta^{D}\rho\leq \frac{C}{\rho}V^{\lambda-\mu}$ in the sense of distributions, where $C$ is a constant and $q-C-1>0$, the
following inequality
\begin{equation}\label{01Sec2thm2-1}\aligned
\int_{M}\frac{|\phi|^{p}}{\rho^{q}}V^{\tau+\lambda-\mu}dv_{g}\leq
\left(\frac{p}{q-C-1}\right)^{q}\left(\int_{M}|\nabla^{D}\phi|^{p}V^{\tau+\lambda-\mu}dv_{g}\right)^{\frac{q}{p}}
\left(\int_{M}|\phi|^{p}V^{\tau+\lambda-\mu}dv_{g}\right)^{\frac{p-q}{p}}
\endaligned\end{equation}
holds.

$\mathrm{(ii)}$\quad When $\Delta^{D}\rho\geq \frac{C}{\rho}V^{\lambda-\mu}$ in the sense of distributions, where $C$ is a constant and $C+1-q>0$, the
following inequality
\begin{equation}\label{02Sec2thm2-2}\aligned
\int_{M}\frac{|\phi|^{p}}{\rho^{q}}V^{\tau+\lambda-\mu}dv_{g}\leq
\left(\frac{p}{C+1-q}\right)^{q}\left(\int_{M}|\nabla^{D}\phi|^{p}V^{\tau+\lambda-\mu}dv_{g}\right)^{\frac{q}{p}}
\left(\int_{M}|\phi|^{p}V^{\tau+\lambda-\mu}dv_{g}\right)^{\frac{p-q}{p}}
\endaligned\end{equation}
holds.
\end{theorem}
\begin{proof}

$\mathrm{(i)}$\quad Let $\mathrm{div}^{D}X=D_{i}X^{i}$ be the divergence of $X$ with respect to $D$, where $X$ is a vector field on $M$. By direct computation, together with assumptions $|\nabla^{D}\rho|=1$ and $\Delta^{D}\rho\leq \frac{C}{\rho}V^{\lambda-\mu}$, we have
\begin{equation}\label{ad01Sec2thm2-1}
\aligned
&V^{\tau}D_{i}\left(|\phi|^{p}\frac{\nabla^{D}_{i}\rho}{\rho^{q-1}}\right)\\
=&\nabla_{i}\left(V^{\tau+\mu-\lambda}|\phi|^{p}\frac{\nabla_{i}\rho}{\rho^{q-1}}\right)\\
=&\nabla_{i}\left(e^{u(\tau+\mu-\lambda)}|\phi|^{p}\frac{\nabla_{i}\rho}{\rho^{q-1}}\right)\\
=&(\tau+\mu-\lambda)\langle\nabla u,\nabla\rho\rangle\frac{|\phi|^{p}}{\rho^{q-1}}V^{\tau+\mu-\lambda}+p\frac{|\phi|^{p-1}}{\rho^{q-1}}\langle\nabla \phi,\nabla\rho\rangle V^{\tau+\mu-\lambda}\\
&+\frac{|\phi|^{p}}{\rho^{q-1}}V^{\tau+\mu-\lambda}\Delta \rho
+(1-q)|\nabla\rho|^{2}\frac{|\phi|^{p}}{\rho^{q}}V^{\tau+\mu-\lambda}\\
=&\frac{|\phi|^{p}}{\rho^{q-1}}V^{\tau}\Delta^{D} \rho+p\frac{|\phi|^{p-1}}{\rho^{q-1}}\langle\nabla^{D} \phi,\nabla^{D}\rho\rangle V^{\tau+\lambda-\mu}+(1-q)|\nabla^{D}\rho|^{2}\frac{|\phi|^{p}}{\rho^{q}}V^{\tau+\lambda-\mu}\\
\leq&(1+C-q)\frac{|\phi|^{p}}{\rho^{q}}V^{\tau+\lambda-\mu}+p\frac{|\phi|^{p-1}}{\rho^{q-1}}\langle\nabla^{D} \phi,\nabla^{D}\rho\rangle V^{\tau+\lambda-\mu}.
\endaligned
\end{equation}
Integrating the above inequality \eqref{ad01Sec2thm2-1} over $M$ yields
\begin{equation}\label{03Sec2thm2-1}\aligned
&(q-C-1)\int_{M}\frac{|\phi|^{p}}{\rho^{q}}V^{\tau+\lambda-\mu} ~dv_{g}\\
\leq& p\int_{M}\frac{|\phi|^{p-1}}{\rho^{q-1}}\langle\nabla^{D} \phi,\nabla^{D}\rho\rangle V^{\tau+\lambda-\mu} ~dv_{g}\\
\leq& p\int_{M}\frac{|\phi|^{p-1}}{\rho^{q-1}}|\nabla^{D} \phi||\nabla^{D}\rho|V^{\tau+\lambda-\mu} ~dv_{g}\\
=& p\int_{M}\frac{|\phi|^{p-1}}{\rho^{q-1}}|\nabla^{D} \phi|V^{\tau+\lambda-\mu} ~dv_{g}.
\endaligned\end{equation}
It follows from H\"{o}lder's inequality that
\begin{equation}\label{04Sec2thm2-1}\aligned
&\int_{M}\frac{|\phi|^{p-1}}{\rho^{q-1}}|\nabla^{D} \phi|V^{\tau+\lambda-\mu} ~dv_{g}\\
=&\int_{M}|\nabla^{D} \phi|V^{\frac{\tau+\lambda-\mu}{p}}\frac{|\phi|^{p-1}}{\rho^{q-1}}V^{\frac{(\tau+\lambda-\mu)(p-1)}{p}} ~dv_{g}\\
\leq&\left(\int_{M}\left[|\nabla^{D} \phi|V^{\frac{\tau+\lambda-\mu}{p}}\right]^{p} ~dv_{g}\right)^{\frac{1}{p}}
\left(\int_{M}\left[\frac{|\phi|^{p-1}}{\rho^{q-1}}V^{\frac{(\tau+\lambda-\mu)(p-1)}{p}}\right]^{\frac{p}{p-1}} ~dv_{g}\right)^{\frac{p-1}{p}}\\
=&\left(\int_{M}|\nabla^{D} \phi|^{p}V^{\tau+\lambda-\mu} ~dv_{g}\right)^{\frac{1}{p}}
\left(\int_{M}\frac{|\phi|^{p}}{\rho^{\frac{p(q-1)}{p-1}}}V^{\tau+\lambda-\mu} ~dv_{g}\right)^{\frac{p-1}{p}}\\
\endaligned\end{equation}
and
\begin{equation}\label{05Sec2thm2-1}\aligned
&\int_{M}\frac{|\phi|^{p}}{\rho^{\frac{p(q-1)}{p-1}}}V^{\tau+\lambda-\mu} ~dv_{g}\\
=&\int_{M}\rho^{-\frac{p(q-1)}{p-1}}\left(|\phi|^{p}V^{\tau+\lambda-\mu}
\right)^{\frac{p(q-1)}{(p-1)q}}
\left(|\phi|^{p}V^{\tau+\lambda-\mu}\right)^{\left(1-\frac{p(q-1)}{(p-1)q}\right)} ~dv_{g}\\
\leq&\left(\int_{M}\left[\frac{\left(|\phi|^{p}V^{\tau+\lambda-\mu}\right)^{\frac{p(q-1)}{(p-1)q}}}
{\rho^{\frac{p(q-1)}{p-1}}}\right]^{\frac{(p-1)q}{p(q-1)}} ~dv_{g}\right)^{\frac{p(q-1)}{(p-1)q}}
\left(\int_{M}\left[\left(|\phi|^{p}V^{\tau+\lambda-\mu}\right)^{\left(1-\frac{p(q-1)}{(p-1)q}\right)}\right]^{\frac{1}{\left(1-\frac{p(q-1)}{(p-1)q}\right)}} ~dv_{g}\right)^{\left(1-\frac{p(q-1)}{(p-1)q}\right)}\\
=&\left(\int_{M}\frac{|\phi|^{p}V^{\tau+\lambda-\mu}}{\rho^{q}} ~dv_{g}\right)^{\frac{p(q-1)}{(p-1)q}}
\left(\int_{M}|\phi|^{p}V^{\tau+\lambda-\mu}~dv_{g}\right)^{\left(1-\frac{p(q-1)}{(p-1)q}\right)}.
\endaligned\end{equation}
Taking \eqref{05Sec2thm2-1} into \eqref{04Sec2thm2-1}, we have
\begin{equation}\label{06Sec2thm2-1}\aligned
&\int_{M}\frac{|\phi|^{p-1}}{\rho^{q-1}}|\nabla^{D} \phi|V^{\tau+\lambda-\mu} ~dv_{g}\\
\leq&\left(\int_{M}|\nabla^{D} \phi|^{p}V^{\tau+\lambda-\mu} ~dv_{g}\right)^{\frac{1}{p}}
\left(\int_{M}\frac{|\phi|^{p}V^{\tau+\lambda-\mu}}{\rho^{q}} ~dv_{g}\right)^{\frac{q-1}{q}}
\left(\int_{M}|\phi|^{p}V^{\tau+\lambda-\mu}~dv_{g}\right)^{\left(1-\frac{p(q-1)}{(p-1)q}\right)\frac{p-1}{p}}.
\endaligned\end{equation}
Then combining \eqref{06Sec2thm2-1} and \eqref{03Sec2thm2-1}, we obtain
$$\aligned
\left(\int_{M}\frac{|\phi|^{p}}{\rho^{q}}V^{\tau+\lambda-\mu} ~dv_{g}\right)^{\frac{1}{q}}
\leq\frac{p}{q-C-1}\left(\int_{M}|\nabla^{D} \phi|^{p}V^{\tau+\lambda-\mu} ~dv_{g}\right)^{\frac{1}{p}}
\left(\int_{M}|\phi|^{p}V^{\tau+\lambda-\mu}~dv_{g}\right)^{\left(1-\frac{p(q-1)}{(p-1)q}\right)\frac{p-1}{p}},
\endaligned$$
which is equivalent to
$$\aligned
\int_{M}\frac{|\phi|^{p}}{\rho^{q}}V^{\tau+\lambda-\mu}dv_{g}\leq
\left(\frac{p}{q-C-1}\right)^{q}\left(\int_{M}|\nabla^{D}\phi|^{p}V^{\tau+\lambda-\mu}dv_{g}\right)^{\frac{q}{p}}
\left(\int_{M}|\phi|^{p}V^{\tau+\lambda-\mu}dv_{g}\right)^{\frac{p-q}{p}}.
\endaligned$$

$\mathrm{(ii)}$\quad By direct computation, together with assumptions $|\nabla^{D}\rho|=1$ and $\Delta^{D}\rho\geq \frac{C}{\rho}V^{\lambda-\mu}$, we have
\begin{equation}\label{ad01Sec2thm2-2}
\aligned
&V^{\tau}D_{i}\left(|\phi|^{p}\frac{\nabla^{D}_{i}\rho}{\rho^{q-1}}\right)\\
=&\nabla_{i}\left(V^{\tau+\mu-\lambda}|\phi|^{p}\frac{\nabla_{i}\rho}{\rho^{q-1}}\right)\\
=&\nabla_{i}\left(e^{u(\tau+\mu-\lambda)}|\phi|^{p}\frac{\nabla_{i}\rho}{\rho^{q-1}}\right)\\
=&(\tau+\mu-\lambda)\langle\nabla u,\nabla\rho\rangle\frac{|\phi|^{p}}{\rho^{q-1}}V^{\tau+\mu-\lambda}+p\frac{|\phi|^{p-1}}{\rho^{q-1}}\langle\nabla \phi,\nabla\rho\rangle V^{\tau+\mu-\lambda}\\
&+\frac{|\phi|^{p}}{\rho^{q-1}}V^{\tau+\mu-\lambda}\Delta \rho
+(1-q)|\nabla\rho|^{2}\frac{|\phi|^{p}}{\rho^{q}}V^{\tau+\mu-\lambda}\\
=&\frac{|\phi|^{p}}{\rho^{q-1}}V^{\tau}\Delta^{D} \rho+p\frac{|\phi|^{p-1}}{\rho^{q-1}}\langle\nabla^{D} \phi,\nabla^{D}\rho\rangle V^{\tau+\lambda-\mu}+(1-q)|\nabla^{D}\rho|^{2}\frac{|\phi|^{p}}{\rho^{q}}V^{\tau+\lambda-\mu}\\
\geq&(1+C-q)\frac{|\phi|^{p}}{\rho^{q}}V^{\tau+\lambda-\mu}+p\frac{|\phi|^{p-1}}{\rho^{q-1}}\langle\nabla^{D} \phi,\nabla^{D}\rho\rangle V^{\tau+\lambda-\mu}.
\endaligned
\end{equation}
Integrating the above inequality \eqref{ad01Sec2thm2-2} over $M$ yields
\begin{equation}\label{03Sec2thm2-2}\aligned
&(C+1-q)\int_{M}\frac{|\phi|^{p}}{\rho^{q}}V^{\tau+\lambda-\mu} ~dv_{g}\\
\leq& -p\int_{M}\frac{|\phi|^{p-1}}{\rho^{q-1}}\langle\nabla^{D} \phi,\nabla^{D}\rho\rangle V^{\tau+\lambda-\mu} ~dv_{g}\\
\leq& \left|-p\int_{M}\frac{|\phi|^{p-1}}{\rho^{q-1}}\langle\nabla^{D} \phi,\nabla^{D}\rho\rangle V^{\tau+\lambda-\mu} ~dv_{g}\right|\\
\leq& p\int_{M}\frac{|\phi|^{p-1}}{\rho^{q-1}}|\nabla^{D} \phi||\nabla^{D}\rho|V^{\tau+\lambda-\mu} ~dv_{g}\\
=& p\int_{M}\frac{|\phi|^{p-1}}{\rho^{q-1}}|\nabla^{D} \phi|V^{\tau+\lambda-\mu} ~dv_{g}.
\endaligned\end{equation}
It follows from H\"{o}lder's inequality that
\begin{equation}\label{04Sec2thm2-2}\aligned
&\int_{M}\frac{|\phi|^{p-1}}{\rho^{q-1}}|\nabla^{D} \phi|V^{\tau+\lambda-\mu} ~dv_{g}
=\int_{M}|\nabla^{D} \phi|V^{\frac{\tau+\lambda-\mu}{p}}\frac{|\phi|^{p-1}}{\rho^{q-1}}|\nabla^{D} \phi|V^{\frac{(\tau+\lambda-\mu)(p-1)}{p}} ~dv_{g}\\
\leq&\left(\int_{M}\left[|\nabla^{D} \phi|V^{\frac{\tau+\lambda-\mu}{p}}\right]^{p} ~dv_{g}\right)^{\frac{1}{p}}
\left(\int_{M}\left[\frac{|\phi|^{p-1}}{\rho^{q-1}}|\nabla^{D} \phi|V^{\frac{(\tau+\lambda-\mu)(p-1)}{p}}\right]^{\frac{p}{p-1}} ~dv_{g}\right)^{\frac{p-1}{p}}\\
=&\left(\int_{M}|\nabla^{D} \phi|^{p}V^{\tau+\lambda-\mu} ~dv_{g}\right)^{\frac{1}{p}}
\left(\int_{M}\frac{|\phi|^{p}}{\rho^{\frac{p(q-1)}{p-1}}}V^{\tau+\lambda-\mu} ~dv_{g}\right)^{\frac{p-1}{p}}\\
\endaligned\end{equation}
and
\begin{equation}\label{05Sec2thm2-2}\aligned
&\int_{M}\frac{|\phi|^{p}}{\rho^{\frac{p(q-1)}{p-1}}}V^{\tau+\lambda-\mu} ~dv_{g}\\
=&\int_{M}\rho^{-\frac{p(q-1)}{p-1}}\left(|\phi|^{p}V^{\tau+\lambda-\mu}
\right)^{\frac{p(q-1)}{(p-1)q}}
\left(|\phi|^{p}V^{\tau+\lambda-\mu}\right)^{\left(1-\frac{p(q-1)}{(p-1)q}\right)} ~dv_{g}\\
\leq&\left(\int_{M}\left[\frac{\left(|\phi|^{p}V^{\tau+\lambda-\mu}\right)^{\frac{p(q-1)}{(p-1)q}}}
{\rho^{\frac{p(q-1)}{p-1}}}\right]^{\frac{(p-1)q}{p(q-1)}} ~dv_{g}\right)^{\frac{p(q-1)}{(p-1)q}}
\left(\int_{M}\left[\left(|\phi|^{p}V^{\tau+\lambda-\mu}\right)^{\left(1-\frac{p(q-1)}{(p-1)q}\right)}\right]^{\frac{1}{\left(1-\frac{p(q-1)}{(p-1)q}\right)}} ~dv_{g}\right)^{\left(1-\frac{p(q-1)}{(p-1)q}\right)}\\
=&\left(\int_{M}\frac{|\phi|^{p}V^{\tau+\lambda-\mu}}{\rho^{q}} ~dv_{g}\right)^{\frac{p(q-1)}{(p-1)q}}
\left(\int_{M}|\phi|^{p}V^{\tau+\lambda-\mu}~dv_{g}\right)^{\left(1-\frac{p(q-1)}{(p-1)q}\right)}.
\endaligned\end{equation}
Taking \eqref{05Sec2thm2-2} into \eqref{04Sec2thm2-2}, we have
\begin{equation}\label{06Sec2thm2-2}\aligned
&\int_{M}\frac{|\phi|^{p-1}}{\rho^{q-1}}|\nabla^{D} \phi|V^{\tau+\lambda-\mu} ~dv_{g}\\
\leq&\left(\int_{M}|\nabla^{D} \phi|^{p}V^{\tau+\lambda-\mu} ~dv_{g}\right)^{\frac{1}{p}}
\left(\int_{M}\frac{|\phi|^{p}V^{\tau+\lambda-\mu}}{\rho^{q}} ~dv_{g}\right)^{\frac{q-1}{q}}
\left(\int_{M}|\phi|^{p}V^{\tau+\lambda-\mu}~dv_{g}\right)^{\left(1-\frac{p(q-1)}{(p-1)q}\right)\frac{p-1}{p}}.
\endaligned\end{equation}
Then combining \eqref{06Sec2thm2-2} and \eqref{03Sec2thm2-2}, we obtain
$$\aligned
\left(\int_{M}\frac{|\phi|^{p}}{\rho^{q}}V^{\tau+\lambda-\mu} ~dv_{g}\right)^{\frac{1}{q}}
\leq\frac{p}{C+1-q}\left(\int_{M}|\nabla^{D} \phi|^{p}V^{\tau+\lambda-\mu} ~dv_{g}\right)^{\frac{1}{p}}
\left(\int_{M}|\phi|^{p}V^{\tau+\lambda-\mu}~dv_{g}\right)^{\left(1-\frac{p(q-1)}{(p-1)q}\right)\frac{p-1}{p}},
\endaligned$$
which is equivalent to
$$\aligned
\int_{M}\frac{|\phi|^{p}}{\rho^{q}}V^{\tau+\lambda-\mu}dv_{g}\leq
\left(\frac{p}{C+1-q}\right)^{q}\left(\int_{M}|\nabla^{D}\phi|^{p}V^{\tau+\lambda-\mu}dv_{g}\right)^{\frac{q}{p}}
\left(\int_{M}|\phi|^{p}V^{\tau+\lambda-\mu}dv_{g}\right)^{\frac{p-q}{p}}.
\endaligned$$
This completes the proof of Theorem \ref{thm2}.
\end{proof}

Under affine connections $D^{\lambda,\mu}$, we now prove an $L^{2}$-version two-weight Hardy type inequality with nonnegative remainder terms.

\begin{theorem}\label{thm4}
Let $(M^{n}, g,V=e^{u})$ be an n-dimensional Riemannian triple and
$\lambda,\mu\in \mathbb{R}$. Let $D=D^{\lambda,\mu}$ be the affine connection defined as in \eqref{01adsec02} and $\tau=(n+1)\lambda+\mu$. Let $a(x)$ and $b(x)$ be nonnegative weight functions in $M$. Let $\omega$ be a positive function satisfying the
differential inequality
\begin{equation}\label{ad0101Sec2thm4}\aligned
-V^{\tau}D_{i}\left(a(x)\nabla^{D}_{i}\omega\right)\geq b(x)\omega V^{\tau+\mu-\lambda}
\endaligned\end{equation}
almost everywhere in $M$. Then the following integral inequality
\begin{equation}\label{01Sec2thm4}\aligned
\int_{M}\big(a(x)|\nabla^{D}\phi|^{2}V^{\tau+\lambda-\mu}-b(x)\phi^{2}V^{\tau+\mu-\lambda}\big)~dv_{g}\geq
\int_{M}a(x)\omega^{2}\left|\nabla^{D}\left(\frac{\phi}{\omega}\right)\right|^{2} V^{\tau+\lambda-\mu}~dv_{g}
\endaligned\end{equation}
holds for any $\phi\in C^{\infty}_{0}(M)$.
\end{theorem}
\begin{proof}
Let $\phi=\omega\varphi$, where $\varphi\in C^{\infty}_{0}(M)$. Simple calculation can be obtained,
\begin{equation}\label{02Sec2thm4}\aligned
&|\nabla\phi|^{2}\\
=&|\omega\nabla\varphi+\varphi\nabla\omega|^{2}\\
=&\omega^{2}|\nabla\varphi|^{2}+\varphi^{2}|\nabla\omega|^{2}
+2\omega\varphi\langle\nabla\varphi,\nabla\omega\rangle\\
=&\omega^{2}|\nabla\varphi|^{2}+\varphi^{2}|\nabla\omega|^{2}+\omega\langle\nabla\omega,\nabla\varphi^{2}\rangle.
\endaligned\end{equation}
Multiplying both sides of \eqref{02Sec2thm4} by $a(x)V^{\tau+\mu-\lambda}$ and applying the integration by parts to the last term on the right hand side of the inequality over $M$ gives
$$\aligned
&\int_{M}a(x)|\nabla\phi|^{2}V^{\tau+\mu-\lambda}~dv_{g}\\
=&\int_{M}a(x)\varphi^{2}|\nabla\omega|^{2}V^{\tau+\mu-\lambda}~dv_{g}+\int_{M}a(x)\omega^{2}|\nabla\varphi|^{2}V^{\tau+\mu-\lambda}~dv_{g}
\\
&+\int_{M}a(x)\omega\langle\nabla\omega,\nabla\varphi^{2}\rangle V^{\tau+\mu-\lambda}~dv_{g}\\
=&\int_{M}a(x)\varphi^{2}|\nabla\omega|^{2}V^{\tau+\mu-\lambda}~dv_{g}+\int_{M}a(x)\omega^{2}|\nabla\varphi|^{2}V^{\tau+\mu-\lambda}~dv_{g}
\\
&-\int_{M}\varphi^{2}\nabla_{i}\left(a(x)\omega V^{\tau+\mu-\lambda}\nabla_{i}\omega\right)~dv_{g}\\
=&-\int_{M}\varphi^{2}\omega\nabla_{i}\left(a(x)V^{\tau+\mu-\lambda}\nabla_{i}\omega\right)~dv_{g}+\int_{M}a(x)\omega^{2}|\nabla\varphi|^{2}V^{\tau+\mu-\lambda}~dv_{g}\\
=&-\int_{M}\varphi^{2}\omega V^{\tau}D_{i}\left(a(x)\nabla^{D}_{i}\omega\right)~dv_{g}+\int_{M}a(x)\omega^{2}|\nabla\varphi|^{2}V^{\tau+\mu-\lambda}~dv_{g},
\endaligned$$
where, in the second and third inequalities, we use the integration by parts \eqref{ad03Sec1}.
Since $-V^{\tau}D_{i}\left(a(x)\nabla^{D}_{i}\omega\right)\geq b(x)\omega V^{\tau+\mu-\lambda}$ then we get
$$\aligned
&\int_{M}a(x)|\nabla\phi|^{2}V^{\tau+\mu-\lambda}~dv_{g}\\
\geq&\int_{M}b(x)\phi^{2}V^{\tau+\mu-\lambda}~dv_{g}+\int_{M}a(x)\omega^{2}|\nabla\varphi|^{2}V^{\tau+\mu-\lambda}~dv_{g}.
\endaligned$$
Since $\varphi=\frac{\phi}{\omega}$, we obtained the desired inequality
\begin{equation}\label{03Sec2thm4}\aligned
&\int_{M}a(x)|\nabla\phi|^{2}V^{\tau+\mu-\lambda}~dv_{g}=\int_{M}a(x)|\nabla^{D}\phi|^{2}V^{\tau+\lambda-\mu}~dv_{g}\\
\geq&\int_{M}b(x)\phi^{2}V^{\tau+\mu-\lambda}~dv_{g}
+\int_{M}a(x)\omega^{2}\left|\nabla\left(\frac{\phi}{\omega}\right)\right|^{2}V^{\tau+\mu-\lambda}~dv_{g}\\
=&\int_{M}b(x)\phi^{2}V^{\tau+\mu-\lambda}~dv_{g}
+\int_{M}a(x)\omega^{2}\left|\nabla^{D}\left(\frac{\phi}{\omega}\right)\right|^{2}V^{\tau+\lambda-\mu}~dv_{g}.
\endaligned\end{equation}
Reorganizing yields the desired equality \eqref{01Sec2thm4}. The proof of Theorem \ref{thm4} is completed.
\end{proof}
We now apply Theorem \ref{thm4} to recover previously known weighted Hardy type inequalities. For instance, by
considering the weight function $a(x)=\rho^{\alpha}$ and $\omega(x)=\rho^{-\left(\frac{C+\alpha-1}{2}\right)}$ in \eqref{01Sec2thm4}, we obtain the following weighted Hardy type inequality:

\begin{corollary}\label{cor01}
Let $(M^{n}, g,V=e^{u})$ be an n-dimensional Riemannian triple and
$\lambda,\mu\in \mathbb{R}$. Let $D=D^{\lambda,\mu}$ be the affine connection defined as in \eqref{01adsec02} and $\tau=(n+1)\lambda+\mu$. Let $\rho$ be a nonnegative function on $M$ such that $|\nabla^{D}\rho|= 1$ and $\Delta^{D}\rho\geq \frac{C}{\rho}V^{\lambda-\mu}
$ in the sense of distribution, where $C>1$, $C+\alpha-1>0$ and $\alpha\in \mathbb{R}$. Then the following inequality
\begin{equation}\label{01Sec2cor1}\aligned
\int_{M}\rho^{\alpha}|\nabla^{D}\phi|^{2}V^{\tau+\lambda-\mu}~dv_{g}\geq\left(\frac{C+\alpha-1}{2}\right)^{2}\int_{M}\rho^{\alpha-2}\phi^{2}V^{\tau+\lambda-\mu}~dv_{g}
\endaligned\end{equation}
holds for any $\phi\in C^{\infty}_{0}(M\setminus\rho^{-1}\{0\})$.
\end{corollary}
\begin{proof}
By a straightforward computation, we obtain
\begin{equation}\label{02Sec2cor1}\aligned
&-V^{\tau}D_{i}\left(a(x)\nabla^{D}_{i}\omega\right)\\
=&-\mathrm{div}\left(V^{\tau+\mu-\lambda}a(x)\nabla\omega\right)\\
=&-\mathrm{div}\left(V^{\tau+\mu-\lambda}\rho^{\alpha}\nabla\rho^{-\left(\frac{C+\alpha-1}{2}\right)}\right)\\
=&-\mathrm{div}\left(V^{\tau+\mu-\lambda}\rho^{\alpha}\left(-\frac{C+\alpha-1}{2}\rho^{-\left(\frac{C+\alpha+1}{2}\right)}\nabla\rho\right)\right)\\
=&\frac{C+\alpha-1}{2}\mathrm{div}\left(V^{\tau+\mu-\lambda}\rho^{\alpha-\frac{C+\alpha+1}{2}}\nabla\rho\right)\\
=&\frac{C+\alpha-1}{2}\left((\tau+\mu-\lambda)V^{\tau+\mu-\lambda}\rho^{\alpha-\frac{C+\alpha+1}{2}}\langle\nabla u,\nabla\rho\rangle+V^{\tau+\mu-\lambda}\rho^{\alpha-\frac{C+\alpha+1}{2}}\Delta\rho\right)\\
&+\frac{C+\alpha-1}{2}\left(\alpha-\frac{C+\alpha+1}{2}\right)V^{\tau+\mu-\lambda}\rho^{\alpha-\frac{C+\alpha+1}{2}-1}|\nabla\rho|^{2}\\
=&\frac{C+\alpha-1}{2}\left(V^{\tau}\rho^{\alpha-\frac{C+\alpha+1}{2}}\Delta^{D}\rho+\left(\alpha-\frac{C+\alpha+1}{2}\right)V^{\tau+\lambda-\mu}\rho^{\alpha-\frac{C+\alpha+1}{2}-1}\right)\\
\geq&\frac{C+\alpha-1}{2}\left(\alpha-\frac{C+\alpha+1}{2}+C\right)V^{\tau+\lambda-\mu}\rho^{\alpha-\frac{C+\alpha+1}{2}-1}\\
=&\left(\frac{C+\alpha-1}{2}\right)^{2}V^{\tau+\lambda-\mu}\rho^{\alpha-2}\rho^{-\left(\frac{C+\alpha-1}{2}\right)}\\
=&\left(\frac{C+\alpha-1}{2}\right)^{2}\rho^{\alpha-2}V^{2(\lambda-\mu)}\omega V^{\tau+\mu-\lambda}.
\endaligned\end{equation}
Since \eqref{ad0101Sec2thm4}, we get
\begin{equation}\label{03Sec2cor1}\aligned
b(x)=\left(\frac{C+\alpha-1}{2}\right)^{2}\rho^{\alpha-2}V^{2(\lambda-\mu)}.
\endaligned\end{equation}
By \eqref{01Sec2thm4} , we have
\begin{equation}\label{04Sec2cor1}\aligned
\int_{M}a(x)|\nabla^{D}\phi|^{2}V^{\tau+\lambda-\mu}~dv_{g}\geq\int_{M}b(x)\phi^{2}V^{\tau+\mu-\lambda}~dv_{g}.
\endaligned\end{equation}
Inserting \eqref{03Sec2cor1} into \eqref{04Sec2cor1}, we obtain
$$\aligned
\int_{M}\rho^{\alpha}|\nabla^{D}\phi|^{2}V^{\tau+\lambda-\mu}~dv_{g}\geq\left(\frac{C+\alpha-1}{2}\right)^{2}\int_{M}\rho^{\alpha-2}\phi^{2}V^{\tau+\lambda-\mu}~dv_{g}.
\endaligned$$
This completes the proof of Corollary \ref{cor01}.
\end{proof}

\begin{remark} \label{Re:1}
Particularly, when $\alpha=\gamma=0$, then our Corollary \ref{cor01} with respect to the affine connection $D^{\lambda,\mu}$ reduces to the Theorem 1.4 of Carron with respect to a Levi-Civita connection of $ g$ in \cite{Carron1997}.
\end{remark}

We now choose the weight function $a(x)=\rho^{\alpha}$ and $\omega(x)=\left(1+\rho^{2}\right)^{-\left(\frac{C+\alpha-1}{2}\right)}$ in \eqref{01Sec2thm4}. Then we obtain the
following weighted Hardy type inequality:

\begin{corollary}\label{cor02}
Under the same conditions of Corollary \ref{cor01}, let $\rho$ be a nonnegative function on $M$ such that $|\nabla^{D}\rho|= 1$ and $\Delta^{D}\rho\geq \frac{C}{\rho}V^{\lambda-\mu}
$ in the sense of
distribution, where $C>1$, $C+\alpha-1>0$ and $\alpha\in \mathbb{R}$. Then the following inequality
\begin{equation}\label{01Sec2cor2}\aligned
\int_{M}\rho^{\alpha}|\nabla^{D}\phi|^{2}V^{\tau+\lambda-\mu}~dv_{g}\geq(C+\alpha-1)(C+\alpha+1)\int_{M}\rho^{\alpha}\frac{|\phi|^{2}}{\left(1+\rho^{2}\right)^{2}}V^{\tau+\lambda-\mu}~dv_{g}
\endaligned\end{equation}
holds for any $\phi\in C^{\infty}_{0}(M\setminus\rho^{-1}\{0\})$.
\end{corollary}
\begin{proof}
By a straightforward computation, we obtain
\begin{equation}\label{02Sec2cor2}\aligned
&-V^{\tau}D_{i}\left(a(x)\nabla^{D}_{i}\omega\right)\\
=&-\mathrm{div}\left(V^{\tau+\mu-\lambda}a(x)\nabla\omega\right)\\
=&-\mathrm{div}\left(V^{\tau+\mu-\lambda}\rho^{\alpha}\nabla\left(1+\rho^{2}\right)^{-\left(\frac{C+\alpha-1}{2}\right)}\right)\\
=&-\mathrm{div}\left(V^{\tau+\mu-\lambda}\rho^{\alpha}\left(-\frac{C+\alpha-1}{2}\left(1+\rho^{2}\right)^{-\left(\frac{C+\alpha+1}{2}\right)-1}2\rho\nabla\rho\right)\right)\\
=&(C+\alpha-1)\mathrm{div}\left(V^{\tau+\mu-\lambda}\rho^{\alpha+1}\left(1+\rho^{2}\right)^{-\left(\frac{C+\alpha-1}{2}\right)-1}\nabla\rho\right)\\
=&(C+\alpha-1)\left(V^{\tau}\rho^{\alpha+1} \left(1+\rho^{2}\right)^{-\left(\frac{C+\alpha-1}{2}\right)-1}\Delta^{D}\rho
+(\alpha+1)V^{\tau+\mu-\lambda}\rho^{\alpha}|\nabla\rho|^{2}\left(1+\rho^{2}\right)^{-\left(\frac{C+\alpha-1}{2}\right)-1}\right)\\
&-(C+\alpha-1)(C+\alpha+1)V^{\tau+\mu-\lambda}\rho^{\alpha+2}|\nabla\rho|^{2}\left(1+\rho^{2}\right)^{-\left(\frac{C+\alpha-1}{2}\right)-2}\\
=&(C+\alpha-1)\left(V^{\tau}\rho^{\alpha+1} \left(1+\rho^{2}\right)^{-\left(\frac{C+\alpha-1}{2}\right)-1}\Delta^{D}\rho
+(\alpha+1)V^{\tau+\lambda-\mu}\rho^{\alpha}|\nabla^{D}\rho|^{2}\left(1+\rho^{2}\right)^{-\left(\frac{C+\alpha-1}{2}\right)-1}\right)\\
&-(C+\alpha-1)(C+\alpha+1)V^{\tau+\lambda-\mu}\rho^{\alpha+2}|\nabla^{D}\rho|^{2}\left(1+\rho^{2}\right)^{-\left(\frac{C+\alpha-1}{2}\right)-2}.
\endaligned\end{equation}
Since $|\nabla^{D}\rho|=1$ and $\Delta^{D}\rho\geq \frac{C}{\rho}V^{\lambda-\mu}
$, we get
\begin{equation}\label{01ad03Sec2cor2}\aligned
&-V^{\tau}D_{i}\left(a(x)\nabla^{D}_{i}\omega\right)\\
\geq&(C+\alpha-1)\left(CV^{\tau+\lambda-\mu}\rho^{\alpha} \left(1+\rho^{2}\right)^{-\left(\frac{C+\alpha-1}{2}\right)-1}
+(\alpha+1)V^{\tau+\lambda-\mu}\rho^{\alpha}\left(1+\rho^{2}\right)^{-\left(\frac{C+\alpha-1}{2}\right)-1}\right)\\
&-(C+\alpha-1)(C+\alpha+1)V^{\tau+\lambda-\mu}\rho^{\alpha+2}\left(1+\rho^{2}\right)^{-\left(\frac{C+\alpha-1}{2}\right)-2}\\
=&(C+\alpha-1)(C+\alpha+1)V^{\tau+\lambda-\mu}\rho^{\alpha} \left(1+\rho^{2}\right)^{-\left(\frac{C+\alpha+1}{2}\right)}\\
&-(C+\alpha-1)(C+\alpha+1)V^{\tau+\lambda-\mu}\rho^{\alpha+2}\left(1+\rho^{2}\right)^{-\left(\frac{C+\alpha-1}{2}\right)-2}\\
=&(C+\alpha-1)(C+\alpha+1)V^{\tau+\lambda-\mu}\rho^{\alpha} \left(1+\rho^{2}\right)^{-\left(\frac{C+\alpha-1}{2}\right)-1}\\
&-(C+\alpha-1)(C+\alpha+1)V^{\tau+\lambda-\mu}\rho^{\alpha}\left(\rho^{2}+1-1\right)\left(1+\rho^{2}\right)^{-\left(\frac{C+\alpha-1}{2}\right)-2}\\
=&(C+\alpha-1)(C+\alpha+1)V^{\tau+\lambda-\mu}\rho^{\alpha}\left(1+\rho^{2}\right)^{-\left(\frac{C+\alpha-1}{2}\right)-2}\\
=&(C+\alpha-1)(C+\alpha+1)V^{\tau+\lambda-\mu}\rho^{\alpha}\left(1+\rho^{2}\right)^{-2}\left(1+\rho^{2}\right)^{-\left(\frac{C+\alpha-1}{2}\right)}\\
=&(C+\alpha-1)(C+\alpha+1)V^{\tau+\lambda-\mu}\rho^{\alpha}\left(1+\rho^{2}\right)^{-2}\omega.
\endaligned\end{equation}
Since \eqref{ad0101Sec2thm4}, we get
\begin{equation}\label{03Sec2cor2}\aligned
b(x)=(C+\alpha-1)(C+\alpha+1)\rho^{\alpha}\left(1+\rho^{2}\right)^{-2}V^{2(\alpha-\gamma)}.
\endaligned\end{equation}
Inserting \eqref{03Sec2cor2} into \eqref{04Sec2cor1}, we obtain
$$\aligned
\int_{M}\rho^{\alpha}|\nabla^{D}\phi|^{2}V^{\tau+\lambda-\mu}~dv_{g}\geq(C+\alpha-1)(C+\alpha+1)\int_{M}\rho^{\alpha}\frac{|\phi|^{2}}{\left(1+\rho^{2}\right)^{2}}V^{\tau+\lambda-\mu}~dv_{g}.
\endaligned$$
This completes the proof of Corollary \ref{cor02}.
\end{proof}

\begin{remark} \label{Re:2}
In particular, if $\alpha=\gamma=0$ and $p=2$, then our Corollary \ref{cor02} with respect to the affine connection $D^{\lambda,\mu}$ reduces to the Corollary 2.3 with respect to a Levi-Civita connection of $ g$ in \cite{Kombe2016}.

The following Hardy type inequality in the Euclidean setting has been proved by Ghoussoub and Moradifam \cite{Ghoussoub2011}:
$$\aligned
\int_{\mathbb{R}}\frac{\left(c+d|x|^{\alpha}\right)^{\beta}}{|x|^{2m}}|\nabla \phi|^{2}~dx \geq\left(\frac{n-2m-2}{2}\right)^{2} \int_{\mathbb{R}^{n}} \frac{\left(c+d|x|^{\alpha}\right)^{\beta}}{|x|^{2m+2}} \phi^{2} ~dx,
\endaligned$$
where $\phi\in C^{\infty}_{0}\left(\mathbb{R}^{n}\right) , c, d> 0, \alpha\beta> 0$ and $m\leq \frac {n-2}{2}.$
\end{remark}

We now choose the pair as
$$\aligned
a\left(x\right)=\frac{\left(c+d\rho^{\alpha}\right)^{\beta}}{\rho^{2m}}\quad\mathrm{and}\quad \omega\left(x\right)=\rho^{-\left(\frac{C-2m-1}{2}\right)}.
\endaligned$$
Then we have the following Hardy type inequality:

\begin{corollary}\label{cor03}
Under the same conditions of Corollary \ref{cor01}, let $\rho$  be a nonnegative function on $M$ such that $|\nabla^{D}\rho|=1$ and $\Delta^{D}\rho\geq \frac{C}{\rho}V^{\lambda-\mu}$ in the sense of distribution, where $C>1$ and $C-2m-1\geq0$. Then the following inequality
\begin{equation}\label{01Sec2cor3}\aligned
\int_{M}\frac{(c+d\rho^{\alpha})^{\beta}}{\rho^{2m}}|\nabla^{D}\phi|^{2}V^{\tau+\lambda-\mu}~dv_{g}\geq\left(\frac{C-2m-1}{2}\right)^{2}\int_{M}\frac{(c+d\rho^{\alpha})^{\beta}}{\rho^{2m+2}}\phi^{2}V^{\tau+\lambda-\mu}~dv_{g}
\endaligned\end{equation}
holds for any $\phi\in C_{0}^{\infty}\left(M\setminus\rho^{-1}\{0\}\right)$. Here $c, d>0, \alpha\beta>0$.
\end{corollary}
\begin{proof}
By a straightforward computation, we obtain
\begin{equation}\label{02Sec2cor3}\aligned
&-V^{\tau}D_{i}\left(a(x)\nabla^{D}_{i}\omega\right)\\
=&-\mathrm{div}\left(V^{\tau+\mu-\lambda}a(x)\nabla\omega\right)\\
=&-\mathrm{div}\left(V^{\tau+\mu-\lambda}\frac{(c+d\rho^{\alpha})^{\beta}}{\rho^{2m}}\nabla\rho^{-\frac{C-2m-1}{2}}\right)\\
=&\mathrm{div}\left(V^{\tau+\mu-\lambda}\frac{(c+d\rho^{\alpha})^{\beta}}{\rho^{2m}}\frac{C-2m-1}{2}\rho^{-\left(\frac{C-2m-1}{2}\right)-1}\nabla\rho\right)\\
=&\frac{C-2m-1}{2}(\tau+\mu-\lambda)V^{\tau+\mu-\lambda}
\frac{(c+d\rho^{\alpha})^{\beta}}{\rho^{2m}}\rho^{-\left(\frac{C-2m-1}{2}\right)-1}\langle\nabla u,\nabla\rho\rangle\\
&+\frac{C-2m-1}{2}\frac{\beta\left(c+d\rho^{\alpha}\right)^{\beta-1}}{\rho^{2m}}\alpha d\rho^{\alpha-1}\rho^{-\left(\frac{C-2m-1}{2}\right)-1}|\nabla\rho|^{2}V^{\tau+\mu-\lambda}\\
&-\frac{2m(C-2m-1)}{2}\frac{(c+d\rho^{\alpha})^{\beta}}{\rho^{2m+1}}|\nabla\rho|^{2}\rho^{-\left(\frac{C-2m-1}{2}\right)-1}V^{\tau+\mu-\lambda}\\
&+\frac{C-2m-1}{2}\frac{(c+d\rho^{\alpha})^{\beta}}{\rho^{2m}}\left(-\frac{C-2m+1}{2}\right)\rho^{-\left(\frac{C-2m-1}{2}\right)-2}|\nabla\rho|^{2}V^{\tau+\mu-\lambda}\\
&+\frac{C-2m-1}{2}\frac{(c+d\rho^{\alpha})^{\beta}}{\rho^{2m}}\rho^{-\left(\frac{C-2m-1}{2}\right)-1}V^{\tau+\mu-\lambda}\Delta\rho\\
=&\frac{C-2m-1}{2}\frac{\beta\left(c+d\rho^{\alpha}\right)^{\beta-1}}{\rho^{2m}}\alpha d\rho^{\alpha-1}\rho^{-\left(\frac{C-2m-1}{2}\right)-1}|\nabla^{D}\rho|^{2}V^{\tau+\lambda-\mu}\\
&-\frac{2m(C-2m-1)}{2}\frac{(c+d\rho^{\alpha})^{\beta}}{\rho^{2m+1}}|\nabla^{D}\rho|^{2}\rho^{-\left(\frac{C-2m-1}{2}\right)-1}V^{\tau+\lambda-\mu}\\
&+\frac{C-2m-1}{2}\frac{(c+d\rho^{\alpha})^{\beta}}{\rho^{2m}}\left(-\frac{C-2m+1}{2}\right)\rho^{-\left(\frac{C-2m-1}{2}\right)-2}|\nabla^{D}\rho|^{2}V^{\tau+\lambda-\mu}\\
&+\frac{C-2m-1}{2}\frac{(c+d\rho^{\alpha})^{\beta}}{\rho^{2m}}\rho^{-\left(\frac{C-2m-1}{2}\right)-1}V^{\tau}\Delta^{D}\rho.
\endaligned\end{equation}
Since $|\nabla^{D}\rho|=1$ and $\Delta^{D}\rho\geq \frac{C}{\rho}V^{\lambda-\mu}
$, we have
\begin{equation}\label{01ad02Sec2cor3}\aligned
&-V^{\tau}D_{i}\left(a(x)\nabla^{D}_{i}\omega\right)\\
\geq&\left(\frac{C-2m-1}{2}\right)^{2}\frac{(c+d\rho^{\alpha})^{\beta}}{\rho^{2m+2}} \rho^{-\left(\frac{C-2m-1}{2}\right)}V^{\tau+\lambda-\mu}.
\endaligned\end{equation}
Since \eqref{ad0101Sec2thm4}, we get
\begin{equation}\label{03Sec2cor3}\aligned
b(x)=\left(\frac{C-2m-1}{2}\right)^{2}\frac{\left(c+d\rho^{\alpha}\right)^{\beta}}{\rho^{2m+2}}V^{2(\alpha-\gamma)}.
\endaligned\end{equation}
Inserting \eqref{03Sec2cor3} into \eqref{04Sec2cor1}, we obtain
$$\aligned
\int_{M}\frac{(c+d\rho^{\alpha})^{\beta}}{\rho^{2m}}|\nabla^{D}\phi|^{2}V^{\tau+\lambda-\mu}~dv_{g}\geq\left(\frac{C-2m-1}{2}\right)^{2}\int_{M}\frac{(c+d\rho^{\alpha})^{\beta}}{\rho^{2m+2}}\phi^{2}V^{\tau+\lambda-\mu}~dv_{g}.
\endaligned$$
This completes the proof of Corollary \ref{cor03}.
\end{proof}

\begin{remark} \label{Re:3}
In particular, if $\alpha=\gamma=0$, then our Corollary \ref{cor03} with respect to the affine connection $D^{\lambda,\mu}$ reduces to the Corollary 2.4 with respect to a Levi-Civita connection of $ g$ in \cite{Kombe2016}.
\end{remark}

Another application of Theorem \ref{thm4} with the special functions $a(x)=\left(1+\rho^{2}\right)^{\alpha}$ and $\omega(x)=\left(1+\rho^{2}\right)^{1-\alpha}$ leads to the following weighted Hardy type inequality:

\begin{corollary}\label{cor04}
Under the same conditions of Corollary \ref{cor01}, let $\rho$  be a nonnegative function on $M$ such that $|\nabla^{D}\rho|= 1$ and $\Delta^{D}\rho\geq \frac{C}{\rho}V^{\lambda-\mu}$ in the sense of distribution, where $C>1$. Then, for any $\phi\in C^{\infty}_{0}\left(M\setminus\rho^{-1}\{0\}\right)$ and $\alpha>1$, we have
\begin{equation}\label{01Sec2cor4}\aligned
\int_{M}\frac{|\nabla^{D}\phi|^{2}}{\left(1+\rho^{2}\right)^{-\alpha}}V^{\tau+\lambda-\mu}~dv_{g}\geq2\left(\alpha-1\right)(C+1)\int_{M}\frac{|\phi|^{2}}{\left(1+\rho^{2}\right)^{1-\alpha}}V^{\tau+\lambda-\mu}~dv_{g}.
\endaligned\end{equation}
\end{corollary}
\begin{proof}
By a straightforward computation, we obtain
\begin{equation}\label{02Sec2cor4}\aligned
&-V^{\tau}D\left(a(x)\nabla^{D}\omega\right)\\
=&-\mathrm{div}\left(V^{\tau+\mu-\lambda}a(x)\nabla\omega\right)\\
=&-\mathrm{div}\left(V^{\tau+\mu-\lambda}\left(1+\rho^{2}\right)^{\alpha}\nabla\left(1+\rho^{2}\right)^{1-\alpha}\right)\\
=&-\mathrm{div}\left(2V^{\tau+\mu-\lambda}\left(1+\rho^{2}\right)^{\alpha}(1-\alpha)\left(1+\rho^{2}\right)^{-\alpha}\rho\nabla\rho\right)\\
=&2(\alpha-1)\mathrm{div}\left(V^{\tau+\mu-\lambda}\rho\nabla\rho\right)\\
=&2(\alpha-1)\left((\tau+\mu-\lambda)V^{\tau+\mu-\lambda}\rho\langle\nabla u,\nabla\rho\rangle+V^{\tau+\mu-\lambda}|\nabla\rho|^{2}+V^{\tau+\mu-\lambda}\rho\Delta\rho\right)\\
=&2(\alpha-1)\left(\rho V^{\tau}\Delta^{D}\rho+V^{\tau+\lambda-\mu}\right)\\
\geq&2(\alpha-1)(C+1)V^{\tau+\lambda-\mu}\\
=&2(\alpha-1)(C+1)\left(1+\rho^{2}\right)^{1-\alpha}\left(1+\rho^{2}\right)^{\alpha-1}V^{\tau+\lambda-\mu}.
\endaligned\end{equation}
Since \eqref{ad0101Sec2thm4}, we get
\begin{equation}\label{03Sec2cor4}\aligned
b(x)=2(\alpha-1)(C+1)\left(1+\rho^{2}\right)^{\alpha-1}V^{2(\lambda-\mu)}.
\endaligned\end{equation}
Inserting \eqref{03Sec2cor4} into \eqref{04Sec2cor1}, we obtain
$$\aligned
\int_{M}\frac{|\nabla^{D}\phi|^{2}}{\left(1+\rho^{2}\right)^{-\alpha}}V^{\tau+\lambda-\mu}~dv_{g}\geq2\left(\alpha-1\right)(C+1)\int_{M}\frac{|\phi|^{2}}{\left(1+\rho^{2}\right)^{1-\alpha}}V^{\tau+\lambda-\mu}~dv_{g}.
\endaligned$$
This completes the proof of Corollary \ref{cor04}.
\end{proof}

\section{Rellich-type inequalities and their proofs}

Some Rellich type inequalities with respect to $D^{\lambda,\mu}$ on  Riemannian manifolds will be derived in this section. We first prove the following weighted Rellich inequality with respect to $D^{\lambda,\mu}$ on $M$.
\begin{theorem}\label{thm6}(Weighted Rellich inequality). Let $(M^{n}, g,V=e^{u})$ be an n-dimensional Riemannian triple and
$\lambda,\mu\in \mathbb{R}$. Let $D=D^{\lambda,\mu}$ be the affine connection defined as in \eqref{01adsec02} and $\tau=(n+1)\lambda+\mu$.
Let $\rho$ be a nonnegative functions on $M$ such that $|\nabla^{D}\rho|=1$ and $\Delta^{D}\rho\geq\frac{C}{\rho}V^{\lambda-\mu}$ in the sense of distribution where $C>1$, $\alpha<2$ and $C+\alpha-3>0$. Then the following inequality holds:
\begin{equation}\label{01Sec3thm6}\aligned
\int_{M}\rho^{\alpha}(\Delta^{D}\phi)^{2}V^{\tau+\mu-\lambda}~dv_{g}\geq\frac{(C+\alpha-3)^{2}(C-\alpha+1)^{2}}{16}\int_{M}\rho^{\alpha-4}\phi^{2}V^{\tau+\lambda-\mu}~dv_{g}
\endaligned\end{equation}
for any $\phi \in C_{0}^{\infty}(M\setminus\rho^{-1}\{0\})$.
\end{theorem}
\begin{proof}
Since $|\nabla^{D}\rho|=1$ and $\Delta^{D}\rho\geq \frac{C}{\rho}V^{\lambda-\mu}$, by direct computation, we have
$$\aligned
V^{\tau}\Delta^{D}\rho^{\alpha-2}
=&V^{\tau}V^{\mu-\lambda}\left(\Delta\rho^{\alpha-2}+(2\mu+n\lambda)\langle\nabla u,\nabla\rho^{\alpha-2}\rangle\right)\\
=&V^{\tau+\mu-\lambda}\left[\mathrm{div}\left((\alpha-2)\rho^{\alpha-3}\nabla\rho\right)+(2\mu+n\lambda)\langle\nabla u,(\alpha-2)\rho^{\alpha-3}\nabla\rho\rangle\right]\\
=&V^{\tau+\mu-\lambda}\left[(\alpha-2)(\alpha-3)\rho^{\alpha-4}|\nabla\rho|^{2}+(\alpha-2)\rho^{\alpha-3}\Delta\rho
+(2\mu+n\lambda)(\alpha-2)\rho^{\alpha-3}\langle\nabla u,\nabla\rho\rangle\right]\\
=&V^{\tau+\lambda-\mu}\left[(\alpha-2)(\alpha-3)\rho^{\alpha-4}|\nabla^{D}\rho|^{2}+(\alpha-2)\rho^{\alpha-3}\Delta^{D}\rho\right]\\
=&(\alpha-2)V^{\tau}\left[(\alpha-3)\rho^{\alpha-4}V^{\lambda-\mu}+\rho^{\alpha-3}\Delta^{D}\rho\right]\\
\leq&(\alpha-2)(\alpha-3+C)\rho^{\alpha-4}V^{\tau+\lambda-\mu}.
\endaligned$$
Multiplying both sides of above inequality by $\phi^{2}$ and integrating over $M$, we can obtain
\begin{equation}\label{02Sec3thm6}\aligned
(\alpha-2)(\alpha-3+C)\int_{M}\rho^{\alpha-4}\phi^{2}V^{\tau+\lambda-\mu} ~dv_{g}\geq\int_{M}\phi^{2}V^{\tau}\Delta^{D}\rho^{\alpha-2} ~dv_{g}.
\endaligned\end{equation}
Using the integration by parts again, we can get
\begin{equation}\label{03Sec3thm6}\aligned
&\int_{M}\phi^{2}V^{\tau}\Delta^{D}\rho^{\alpha-2} ~dv_{g}\\
=&-\int_{M}\langle\nabla\phi^{2},\nabla^{D}\rho^{\alpha-2}\rangle V^{\tau} ~dv_{g}\\
=&\int_{M}\rho^{\alpha-2}V^{\tau}\Delta^{D}\phi^{2} ~dv_{g}\\
=&\int_{M}V^{\mu-\lambda}\rho^{\alpha-2}\left(\Delta\phi^{2}+(2\mu+n\lambda)\langle\nabla u,\nabla\phi^{2}\rangle\right)V^{\tau} ~dv_{g}\\
=&\int_{M}V^{\mu-\lambda}\rho^{\alpha-2}\left[2\left(\phi\Delta\phi+|\nabla\phi|^{2}\right)+(2\mu+n\lambda)2\phi\langle\nabla u,\nabla\phi\rangle\right]V^{\tau} ~dv_{g}\\
=&2\int_{M}\rho^{\alpha-2}\left(\phi\Delta^{D}\phi+|\nabla^{D}\phi|^{2}V^{\lambda-\mu}\right)V^{\tau} ~dv_{g}.
\endaligned\end{equation}
Taking \eqref{03Sec3thm6} into \eqref{02Sec3thm6}, we obtain
\begin{equation}\label{04Sec3thm6}\aligned
-\int_{M}\rho^{\alpha-2}\left(\phi\Delta^{D}\phi\right)V^{\tau} ~dv_{g}
\geq&\int_{M}\rho^{\alpha-2}|\nabla^{D}\phi|^{2}V^{\tau+\lambda-\mu} ~dv_{g}\\
&-\frac{(\alpha-2)(\alpha-3+C)}{2}\int_{M}\rho^{\alpha-4}\phi^{2}V^{\tau+\lambda-\mu} ~dv_{g}.
\endaligned\end{equation}
By \eqref{01Sec2cor1} in Corollary \ref{cor01}, we have
\begin{equation}\label{05Sec3thm6}\aligned
\int_{M}\rho^{\alpha-2}|\nabla^{D}\phi|^{2}V^{\tau+\lambda-\mu}~dv_{g}\geq\left(\frac{C+\alpha-3}{2}\right)^{2}\int_{M}\rho^{\alpha-4}\phi^{2}V^{\tau+\lambda-\mu}~dv_{g}
\endaligned\end{equation}
Inserting \eqref{05Sec3thm6} into \eqref{04Sec3thm6}, we have
\begin{equation}\label{06Sec3thm6}\aligned
&-\int_{M}\rho^{\alpha-2}\left(\phi\Delta^{D}\phi\right)V^{\tau} ~dv_{g}\\
\geq&\frac{\left(C+\alpha-3\right)^{2}}{2}\int_{M}\rho^{\alpha-4}\phi^{2}V^{\tau+\lambda-\mu}~dv_{g}
-\frac{(\alpha-2)(\alpha-3+C)}{2}\int_{M}\rho^{\alpha-4}\phi^{2}V^{\tau+\lambda-\mu} ~dv_{g}\\
=&\frac{\left(C+\alpha-3\right)\left(C-\alpha+1\right)}{2}\int_{M}\rho^{\alpha-4}\phi^{2}V^{\tau+\lambda-\mu}~dv_{g}.
\endaligned\end{equation}
Using Young's inequality in \eqref{06Sec3thm6}, we have
\begin{equation}\label{07Sec3thm6}\aligned
-\int_{M}\rho^{\alpha-2}\left(\phi\Delta^{D}\phi\right)V^{\tau} ~dv_{g}
\leq&\frac{\left(C+\alpha-3\right)\left(C-\alpha+1\right)}{4}\int_{M}\rho^{\alpha-4}\phi^{2}V^{\tau+\lambda-\mu}~dv_{g}\\
&+\frac{4}{\left(C+\alpha-3\right)\left(C-\alpha+1\right)}\int_{M}\rho^{\alpha}\left(\Delta^{D}\phi\right)^{2}V^{\tau+\mu-\lambda}~dv_{g}.
\endaligned\end{equation}
Then, combining \eqref{06Sec3thm6} and \eqref{07Sec3thm6}, we know that
$$\aligned
&\frac{4}{\left(C+\alpha-3\right)\left(C-\alpha+1\right)}\int_{M}\rho^{\alpha}\left(\Delta^{D}\phi\right)^{2}V^{\tau+\mu-\lambda}~dv_{g}\\
\geq&\frac{\left(C+\alpha-3\right)\left(C-\alpha+1\right)}{4}\int_{M}\rho^{\alpha-4}\phi^{2}V^{\tau+\lambda-\mu}~dv_{g}.
\endaligned$$
Thus, we have
\begin{equation}\label{08Sec3thm6}\aligned
\int_{M}\rho^{\alpha}\left(\Delta^{D}\phi\right)^{2}V^{\tau+\mu-\lambda}~dv_{g}
\geq\frac{\left(C+\alpha-3\right)^{2}\left(C-\alpha+1\right)^{2}}{16}\int_{M}\rho^{\alpha-4}\phi^{2}V^{\tau+\lambda-\mu}~dv_{g}.
\endaligned\end{equation}
\end{proof}

Then, we would like to give two weighted Rellich type inequalities with respect to $D^{\lambda,\mu}$ on $M$, which connect the first and the second order derivatives, as follows.
\begin{theorem}\label{thm7}
Let $(M^{n}, g,V=e^{u})$ be an n-dimensional Riemannian triple and
$\lambda,\mu\in \mathbb{R}$. Let $D=D^{\lambda,\mu}$ be the affine connection defined as in \eqref{01adsec02} and $\tau=(n+1)\lambda+\mu$.
Let $\rho$ be a nonnegative function on $M$ such that $|\nabla^{D}\rho|=1$ in the sense of distributions.
Then for any $\phi\in C^{\infty}_{0} (M\backslash\rho^{-1}\{0\})$, we have

$\mathrm{(i)}$\quad When $\Delta^{D}\rho\leq \frac{C}{\rho}V^{\lambda-\mu}$ in the sense of distributions, where $C<1$ is a constant,
for $2<\alpha<3-C$, the following inequality
\begin{equation}\label{01Sec3thm7-1}\aligned
\int_{M}\rho^{\alpha}\left(\Delta^{D}\phi\right)^{2}V^{\tau+\mu-\lambda} ~dv_{g}
\geq2(\alpha-2)(3-\alpha-C)
\int_{M}\rho^{\alpha-2}|\nabla^{D}\phi|^{2}V^{\tau+\lambda-\mu} ~dv_{g}
\endaligned\end{equation}
holds.

$\mathrm{(ii)}$\quad When $\Delta^{D}\rho\geq \frac{C}{\rho}V^{\lambda-\mu}$ in the sense of distributions, where $C>1$ is a constant,
for $\frac{7-C}{3}<\alpha<2$, the following inequality
\begin{equation}\label{02Sec3thm7-2}\aligned
\int_{M}\rho^{\alpha}\left(\Delta^{D}\phi\right)^{2}V^{\tau+\mu-\lambda} ~dv_{g}
\geq\frac{(C-\alpha+1)^{2}}{4}
\int_{M}\rho^{\alpha-2}|\nabla^{D}\phi|^{2}V^{\tau+\lambda-\mu} ~dv_{g}
\endaligned\end{equation}
holds.
\end{theorem}
\begin{proof}
$\mathrm{(i)}$\quad Since $\alpha-2>0$, $|\nabla^{D}\rho|=1$ and $\Delta^{D}\rho\leq \frac{C}{\rho}V^{\lambda-\mu}$, by direct computation, we have
$$\aligned
\Delta^{D}\left(\rho^{\alpha-2}\right)V^{\tau}
=&V^{\tau}V^{\mu-\lambda}\left(\Delta\rho^{\alpha-2}+(2\mu+n\lambda)\langle\nabla u,\nabla\rho^{\alpha-2}\rangle\right)\\
=&V^{\tau+\mu-\lambda}\left[\mathrm{div}\left((\alpha-2)\rho^{\alpha-3}\nabla\rho\right)+(2\mu+n\lambda)\langle\nabla u,(\alpha-2)\rho^{\alpha-3}\nabla\rho\rangle\right]\\
=&V^{\tau+\mu-\lambda}\left[(\alpha-2)(\alpha-3)\rho^{\alpha-4}|\nabla\rho|^{2}+(\alpha-2)\rho^{\alpha-3}\Delta\rho
+(2\mu+n\lambda)(\alpha-2)\rho^{\alpha-3}\langle\nabla u,\nabla\rho\rangle\right]\\
=&V^{\tau+\lambda-\mu}\left[(\alpha-2)(\alpha-3)\rho^{\alpha-4}|\nabla^{D}\rho|^{2}+(\alpha-2)\rho^{\alpha-3}\Delta^{D}\rho\right]\\
=&(\alpha-2)V^{\tau}\left[(\alpha-3)\rho^{\alpha-4}V^{\lambda-\mu}+\rho^{\alpha-3}\Delta^{D}\rho\right]\\
\leq&(\alpha-2)(\alpha-3+C)\rho^{\alpha-4}V^{\tau+\lambda-\mu}.
\endaligned$$
Multiplying both sides of above inequality by $\phi^{2}$ and integrating over $M$, we can obtain
\begin{equation}\label{02Sec3thm7-1}\aligned
(\alpha-2)(\alpha-3+C)\int_{M}\rho^{\alpha-4}\phi^{2}V^{\tau+\lambda-\mu} ~dv_{g}\geq\int_{M}\phi^{2}\Delta^{D}\left(\rho^{\alpha-2}\right)V^{\tau} ~dv_{g}.
\endaligned\end{equation}
From similar calculations in Theorem \ref{thm6}, we can get the following result:
\begin{equation}\label{03Sec3thm7-1}\aligned
\int_{M}\phi^{2}\Delta^{D}\left(\rho^{\alpha-2}\right)V^{\tau} ~dv_{g}
=\int_{M}2\rho^{\alpha-2}\left(2\phi\Delta^{D}\phi+|\nabla^{D}\phi|^{2}V^{\lambda-\mu}\right)V^{\tau} ~dv_{g}.
\endaligned\end{equation}
Taking \eqref{03Sec3thm7-1} into \eqref{02Sec3thm7-1}, we obtain
\begin{equation}\label{04Sec3thm7-1}\aligned
-2\int_{M}\rho^{\alpha-2}\left(\phi\Delta^{D}\phi\right)V^{\tau} ~dv_{g}
\geq&2\int_{M}\rho^{\alpha-2}|\nabla^{D}\phi|^{2}V^{\tau+\lambda-\mu} ~dv_{g}\\
&-(\alpha-2)(\alpha-3+C)\int_{M}\rho^{\alpha-4}\phi^{2}V^{\tau+\lambda-\mu} ~dv_{g}.
\endaligned\end{equation}

On the other hand, using Young's inequality, we have
\begin{equation}\label{05Sec3thm7-1}\aligned
-2\int_{M}\rho^{\alpha-2}\left(\phi\Delta^{D}\phi\right)V^{\tau} ~dv_{g}
\leq&(\alpha-2)(3-\alpha-C)\int_{M}\rho^{\alpha-4}\phi^{2}V^{\tau+\lambda-\mu} ~dv_{g}\\
&+\frac{1}{(\alpha-2)(3-\alpha-C)}\int_{M}\rho^{\alpha}\left(\Delta^{D}\phi\right)^{2}V^{\tau+\mu-\lambda} ~dv_{g}.
\endaligned\end{equation}
Then, combining \eqref{04Sec3thm7-1} and \eqref{05Sec3thm7-1}, we know that \eqref{01Sec3thm7-1} is true.

$\mathrm{(ii)}$\quad From similar calculations in Theorem \ref{thm6}, we can get the following result:
\begin{equation}\label{04Sec3thm7-2}\aligned
-\int_{M}\rho^{\alpha-2}\left(\phi\Delta^{D}\phi\right)V^{\tau} ~dv_{g}
\geq&\int_{M}\rho^{\alpha-2}|\nabla^{D}\phi|^{2}V^{\tau+\lambda-\mu} ~dv_{g}\\
&-\frac{(\alpha-2)(\alpha-3+C)}{2}\int_{M}\rho^{\alpha-4}\phi^{2}V^{\tau+\lambda-\mu} ~dv_{g},
\endaligned\end{equation}
where $\frac{7-C}{3}<\alpha<2$. Using Young's inequality in \eqref{04Sec3thm7-2}, we have
\begin{equation}\label{05Sec3thm7-2}\aligned
-\int_{M}\rho^{\alpha-2}\left(\phi\Delta^{D}\phi\right)V^{\tau} ~dv_{g}
\leq&\varepsilon\int_{M}\rho^{\alpha-4}\phi^{2}V^{\tau+\lambda-\mu} ~dv_{g}\\
&+\frac{1}{4\varepsilon}\int_{M}\rho^{\alpha}\left(\Delta^{D}\phi\right)^{2}V^{\tau+\mu-\lambda} ~dv_{g},
\endaligned\end{equation}
where $\varepsilon$ is a constant to be determined and $\varepsilon>0$.
Substituting \eqref{05Sec3thm7-2} into \eqref{04Sec3thm7-2}, we get
\begin{equation}\label{06Sec3thm7-2}\aligned
\int_{M}\rho^{\alpha-2}|\nabla^{D}\phi|^{2}V^{\tau+\lambda-\mu} ~dv_{g}
\leq& B\int_{M}\rho^{\alpha-4}\phi^{2}V^{\tau+\lambda-\mu} ~dv_{g}\\
&+\frac{1}{4\varepsilon}\int_{M}\rho^{\alpha}\left(\Delta^{D}\phi\right)^{2}V^{\tau+\mu-\lambda} ~dv_{g},
\endaligned\end{equation}
where $B=\varepsilon+\frac{(\alpha-2)(\alpha-3+C)}{2}$.

According to the sign of $B$, we need to consider the following three cases:\\
{\bf Case 1:}

We choose $B<0$ and $\varepsilon=\frac{(C-\alpha+1)^{2}}{16}$, then \eqref{06Sec3thm7-2} becomes
$$\aligned
\int_{M}\rho^{\alpha-2}|\nabla^{D}\phi|^{2}V^{\tau+\lambda-\mu} ~dv_{g}
\leq&\frac{1}{4\varepsilon}\int_{M}\rho^{\alpha}\left(\Delta^{D}\phi\right)^{2}V^{\tau+\mu-\lambda} ~dv_{g},
\endaligned$$
which is equivalent to
$$\aligned
\int_{M}\rho^{\alpha}\left(\Delta^{D}\phi\right)^{2}V^{\tau+\mu-\lambda} ~dv_{g}
\geq&4\varepsilon\int_{M}\rho^{\alpha-2}|\nabla^{D}\phi|^{2}V^{\tau+\lambda-\mu} ~dv_{g}\\
=&\frac{(C-\alpha+1)^{2}}{4}
\int_{M}\rho^{\alpha-2}|\nabla^{D}\phi|^{2}V^{\tau+\lambda-\mu} ~dv_{g}.
\endaligned$$
{\bf Case 2:}

We choose $B=0$ and $\varepsilon=-\frac{(C+\alpha-3)(\alpha-2)}{2}$, then \eqref{06Sec3thm7-2} becomes
\begin{equation}\label{01ad06Sec3thm7-2}\aligned
\int_{M}\rho^{\alpha-2}|\nabla^{D}\phi|^{2}V^{\tau+\lambda-\mu} ~dv_{g}
\leq&\frac{1}{4\varepsilon}\int_{M}\rho^{\alpha}\left(\Delta^{D}\phi\right)^{2}V^{\tau+\mu-\lambda} ~dv_{g}\\
=&\frac{1}{2(2-\alpha)(C+\alpha-3)}\int_{M}\rho^{\alpha}\left(\Delta^{D}\phi\right)^{2}V^{\tau+\mu-\lambda} ~dv_{g}.
\endaligned\end{equation}
In \eqref{01ad06Sec3thm7-2}, since $C>1$ and $\frac{7-C}{3}<\alpha<2$, we can choose $C$ close enough to $1$ and $\alpha$ close enough to $2$ such that
$$\aligned
\int_{M}\rho^{\alpha}\left(\Delta^{D}\phi\right)^{2}V^{\tau+\mu-\lambda} ~dv_{g}
\geq&2(2-\alpha)(C+\alpha-3)\int_{M}\rho^{\alpha-2}|\nabla^{D}\phi|^{2}V^{\tau+\lambda-\mu} ~dv_{g}\\
\geq&\frac{(C-\alpha+1)^{2}}{4}
\int_{M}\rho^{\alpha-2}|\nabla^{D}\phi|^{2}V^{\tau+\lambda-\mu} ~dv_{g}.
\endaligned$$
{\bf Case 3:}

We choose $B>0$, using the weighted Rellich inequality \eqref{01Sec3thm6}, then \eqref{06Sec3thm7-2} becomes
$$\aligned
&\int_{M}\rho^{\alpha-2}|\nabla^{D}\phi|^{2}V^{\tau+\lambda-\mu} ~dv_{g}\\
\leq&\left(\varepsilon+\frac{(C+\alpha-3)(\alpha-2)}{2}\right)\left(\frac{16}{(C+\alpha-3)^{2}(C-\alpha+1)^{2}}\right)
\int_{M}\rho^{\alpha}\left(\Delta^{D}\phi\right)^{2}V^{\tau+\mu-\lambda} ~dv_{g}\\
&+\frac{1}{4\varepsilon}\int_{M}\rho^{\alpha}\left(\Delta^{D}\phi\right)^{2}V^{\tau+\mu-\lambda} ~dv_{g}\\
=&P_{C,\alpha}(\varepsilon)
\int_{M}\rho^{\alpha}\left(\Delta^{D}\phi\right)^{2}V^{\tau+\mu-\lambda} ~dv_{g},
\endaligned$$
where $P_{C,\alpha}(\varepsilon)=\frac{16\varepsilon}{(C+\alpha-3)^{2}(C-\alpha+1)^{2}}+\frac{8(\alpha-2)}{(C+\alpha-3)(C-\alpha+1)^{2}}+\frac{1}{4\varepsilon}$.
Note that the function $P_{C,\alpha}(\varepsilon)$ attains the minimum for $\epsilon=\frac{(C+\alpha-3)(C-\alpha+1)}{8}$, and this minimum is equal to $\frac{4}{(C-\alpha+1)^{2}}$. Therefore we have the following inequality:
$$\aligned
\int_M\rho^\alpha\left(\Delta^{D}\phi\right)^{2}V^{\tau+\mu-\lambda} ~dv_{g}\geq\frac{(C-\alpha+1)^2}{4}\int_M\rho^{\alpha-2}|\nabla^{D}\phi|^2V^{\tau+\lambda-\mu} ~dv_{g}.
\endaligned$$
This completes the proof of Theorem \ref{thm7}.
\end{proof}

\section{Several inequalities of other types}

In the last section, by applying some results we have derived in the above two sections, we can prove
some inequalities of other types with respect to $D^{\lambda,\mu}$ on $M$. First, we can give the following Hardy-Poincar\'{e} type inequalities with respect to $D^{\lambda,\mu}$.

\begin{theorem}\label{thm9}
Let $(M^{n}, g,V=e^{u})$ be an n-dimensional Riemannian triple and
$\lambda,\mu\in \mathbb{R}$. Let $D=D^{\lambda,\mu}$ be the affine connection defined as in \eqref{01adsec02} and $\tau=(n+1)\lambda+\mu$.
Let $\rho$ be a nonnegative function on $M$ such that $|\nabla^{D}\rho|=1$ in the sense of distributions.
Then for any $\phi\in C^{\infty}_{0} (M\backslash\rho^{-1}\{0\})$, we have

$\mathrm{(i)}$\quad When $\Delta^{D}\rho\leq \frac{C}{\rho}V^{\lambda-\mu}$ in the sense of distributions, where $C>0$ is a constant.
Let $1<p<+\infty$ and $C+\alpha+1<0$, then the following inequality
\begin{equation}\label{01Sec2thm9-1}\aligned
\int_{M}\rho^{\alpha+p}|\langle\nabla^{D}\rho, \nabla^{D}\phi\rangle|^{p}V^{\tau+(1-2p)(\mu-\lambda)} ~dv_{g}\geq \left(\frac{|C+\alpha+1|}{p}\right)^{p}\int_{M}\rho^{\alpha}|\phi|^{p}V^{\tau+\mu-\lambda} ~dv_{g}
\endaligned\end{equation}
holds.

$\mathrm{(ii)}$\quad When $\Delta^{D}\rho\geq \frac{C}{\rho}V^{\lambda-\mu}$ in the sense of distributions, where $C>0$ is a constant.
Let $1<p<+\infty$ and $C+\alpha+1>0$, then the following inequality
\begin{equation}\label{01Sec2thm9-2}\aligned
\int_{M}\rho^{\alpha+p}|\langle\nabla^{D}\rho, \nabla^{D}\phi\rangle|^{p}V^{\tau+(1-2p)(\mu-\lambda)} ~dv_{g}\geq \left(\frac{C+\alpha+1}{p}\right)^{p}\int_{M}\rho^{\alpha}|\phi|^{p}V^{\tau+\mu-\lambda} ~dv_{g}
\endaligned\end{equation}
holds.
\end{theorem}
\begin{proof}

$\mathrm{(i)}$\quad  It follows from $|\nabla^{D}\rho|=1$ and $\Delta^{D}\rho\leq \frac{C}{\rho}V^{\lambda-\mu}$ that
\begin{equation}\label{02Sec2thm9-1}\aligned
&V^{\tau}\rm{div}^{D}\left(\rho\nabla^{D}\rho\right)\\
=&\rm{div}(V^{\tau}\rho\nabla^{D}\rho)\\
=&\rm{div}(V^{\tau+\mu-\lambda}\rho\nabla\rho)\\
=&V^{\tau+\mu-\lambda}|\nabla\rho|^{2}+\rho(\tau+\mu-\lambda)V^{\tau+\mu-\lambda}\langle\nabla u,\nabla\rho\rangle+\rho V^{\tau+\mu-\lambda}\Delta\rho\\
=&V^{\tau+\lambda-\mu}|\nabla^{D}\rho|^{2}+\rho V^{\tau}\Delta^{D}\rho\\
\leq&(1+C)V^{\tau+\lambda-\mu}.
\endaligned\end{equation}
Multiplying \eqref{02Sec2thm9-1} by $\rho^{\alpha}|\phi|^{p}$ and integrating over $M$ yields
\begin{equation}\label{03Sec2thm9-1}\aligned
&(1+C)\int_{M}\rho^{\alpha}|\phi|^{p}V^{\tau+\lambda-\mu} ~dv_{g}\\
\geq& \int_{M}V^{\tau}\rm{div}^{D}(\rho\nabla^{D}\rho)\rho^{\alpha}|\phi|^{p} ~dv_{g} \\
=&\int_{M}\rm{div}(\rho V^{\tau+\mu-\lambda}\nabla\rho)\rho^{\alpha}|\phi|^{p} ~dv_{g} \\
=&-\int_{M}V^{\tau+\mu-\lambda}\langle\rho\nabla\rho,\nabla(\rho^{\alpha}|\phi|^{p})\rangle ~dv_{g}\\
=&-\int_{M}V^{\tau+\mu-\lambda}\langle\rho\nabla\rho,\alpha\rho^{\alpha-1}\nabla\rho|\phi|^{p}+p|\phi|^{p-2}\phi\rho^{\alpha}\nabla\phi
\rangle ~dv_{g}\\
=&-\alpha\int_{M}V^{\tau+\mu-\lambda}\rho^{\alpha}|\nabla\rho|^{2}|\phi|^{p} ~dv_{g}-p\int_{M}\rho^{\alpha+1}|\phi|^{p-2}\phi\langle\nabla\rho,\nabla\phi\rangle V^{\tau+\mu-\lambda} ~dv_{g}\\
=&-\alpha\int_{M}V^{\tau+\lambda-\mu}\rho^{\alpha}|\phi|^{p} ~dv_{g}-p\int_{M}\rho^{\alpha+1}|\phi|^{p-2}\phi\langle\nabla\rho,\nabla\phi\rangle V^{\tau+\mu-\lambda} ~dv_{g}.
\endaligned\end{equation}
Since $C+\alpha+1<0$, we can infer from \eqref{03Sec2thm9-1} that
$$\aligned
&(C+\alpha+1)\int_{M}\rho^{\alpha}|\phi|^{p}V^{\tau+\lambda-\mu} ~dv_{g}\geq
-p\int_{M}\rho^{\alpha+1}|\phi|^{p-2}\phi\langle\nabla\rho,\nabla\phi\rangle V^{\tau+\mu-\lambda} ~dv_{g},
\endaligned$$
which is equivalent to
\begin{equation}\label{04Sec2thm9-1}\aligned
-(C+\alpha+1)\int_{M}\rho^{\alpha}|\phi|^{p}V^{\tau+\lambda-\mu} ~dv_{g}\leq
p\int_{M}V^{\tau+\mu-\lambda}\rho^{\alpha+1}|\phi|^{p-2}\phi\langle\nabla\rho,\nabla\phi\rangle ~dv_{g}.
\endaligned\end{equation}
By \eqref{04Sec2thm9-1}, we can get
\begin{equation}\label{0qad04Sec2thm9-1}\aligned
&|C+\alpha+1|\int_{M}\rho^{\alpha}|\phi|^{p}V^{\tau+\lambda-\mu} ~dv_{g}\\
\leq&
p\left|\int_{M}\rho^{\alpha+1}|\phi|^{p-2}\phi\langle\nabla\rho,\nabla\phi\rangle V^{\tau+\mu-\lambda} ~dv_{g}\right|\\
\leq&p\int_{M}\rho^{\alpha+1}|\phi|^{p-2}|\phi|\left|\langle\nabla\rho,\nabla\phi\rangle\right| V^{\tau+\mu-\lambda} ~dv_{g}\\
=&p\int_{M}\rho^{\alpha+1}|\phi|^{p-1}\left|\langle\nabla\rho,\nabla\phi\rangle\right| V^{\tau+\mu-\lambda} ~dv_{g}.
\endaligned\end{equation}
Applying H\"{o}lder's inequality shows that
\begin{equation}\label{005Sec2thm9-1}\aligned
&p\int_{M}\rho^{\alpha+1}|\phi|^{p-1}\left|\langle\nabla\rho,\nabla\phi\rangle\right| V^{\tau+\mu-\lambda} ~dv_{g}\\
\leq&p\left[\int_{M}\left(\rho^{\frac{\alpha(p-1)}{p}}|\phi|^{p-1}V^{\frac{p-1}{p}(\tau+\lambda-\mu)}\right)^{\frac{p}{p-1}} ~dv_{g}\right]^{\frac{p-1}{p}}
\left[\int_{M}\left(\rho^{\frac{\alpha+p}{p}}\left|\langle\nabla\rho,\nabla\phi\rangle\right|V^{\frac{1}{p}\tau+\left(\frac{1}{p}-2\right)(\lambda-\mu)}\right)^{p} ~dv_{g}\right]^{\frac{1}{p}}\\
=&p\left(\int_{M}\rho^{\alpha}|\phi|^{p}V^{\tau+\lambda-\mu} ~dv_{g}\right)^{\frac{p-1}{p}}\left(\int_{M}\rho^{\alpha+p}|\langle\nabla\rho,\nabla\phi\rangle|^{p}V^{\tau+\left(1-2p\right)(\lambda-\mu)} ~dv_{g}\right)^{\frac{1}{p}}.
\endaligned\end{equation}
By Young's inequality, we have
\begin{equation}\label{006Sec2thm9-1}\aligned
&\left(\int_{M}\rho^{\alpha}|\phi|^{p}V^{\tau+\lambda-\mu} ~dv_{g}\right)^{\frac{p-1}{p}}\left(\int_{M}\rho^{\alpha+p}|\langle\nabla\rho,\nabla\phi\rangle|^{p}V^{\tau+\left(1-2p\right)(\lambda-\mu)} ~dv_{g}\right)^{\frac{1}{p}} \\
\leq&\frac{p-1}{p}\varepsilon^{-\frac{p}{p-1}}\int_{M}\rho^{\alpha}|\phi|^{p}V^{\tau+\lambda-\mu} ~dv_{g}
+\frac{\varepsilon^{p}}{p}\int_{M}\rho^{\alpha+p}|\langle\nabla\rho,\nabla\phi\rangle|^{p}V^{\tau+\left(1-2p\right)(\lambda-\mu)} ~dv_{g} \endaligned\end{equation}
for any $\varepsilon>0$. Inserting \eqref{005Sec2thm9-1} and \eqref{006Sec2thm9-1} into \eqref{0qad04Sec2thm9-1}, we obtain
\begin{equation}\label{05Sec2thm9-1}\aligned
|C+\alpha+1|\int_{M}\rho^{\alpha}|\phi|^{p}V^{\tau+\lambda-\mu} ~dv_{g}
\leq&(p-1)\varepsilon^{-\frac{p}{p-1}}\int_{M}\rho^{\alpha}|\phi|^{p}V^{\tau+\lambda-\mu} ~dv_{g}\\
&+\varepsilon^{p}\int_{M}\rho^{\alpha+p}|\langle\nabla\rho,\nabla\phi\rangle|^{p}V^{\tau+\left(1-2p\right)(\lambda-\mu)} ~dv_{g}. \endaligned\end{equation}
Thus, for any $\varepsilon>0$, from \eqref{05Sec2thm9-1}, we have
$$\aligned
&\int_{M}\rho^{\alpha+p}\left|\langle\nabla^{D}\rho,\nabla^{D}\phi\rangle\right|^{p}V^{\tau+(1-2p)(\mu-\lambda)} ~dv_{g}\\
\geq&\varepsilon^{-p}\left(|C+\alpha+1|-(p-1)\varepsilon^{-\frac{p}{p-1}}\right)\int_{M}\rho^{\alpha}|\phi|^{p}V^{\tau+\lambda-\mu} ~dv_{g}.
\endaligned$$
Taking
$$\aligned
\varepsilon=\left(\frac{p}{|C+\alpha+1|}\right)^{\frac{p-1}{p}}
\endaligned$$
in the above inequality, we can get \eqref{01Sec2thm9-1} directly.

$\mathrm{(ii)}$\quad  It follows from $|\nabla^{D}\rho|=1$ and $\Delta^{D}\rho\geq \frac{C}{\rho}V^{\lambda-\mu}$ that
\begin{equation}\label{02Sec2thm9-2}\aligned
&V^{\tau}\rm{div}^{D}\left(\rho\nabla^{D}\rho\right)\\
=&\rm{div}(V^{\tau}\rho\nabla^{D}\rho)\\
=&\rm{div}(V^{\tau+\mu-\lambda}\rho\nabla\rho)\\
=&V^{\tau+\mu-\lambda}|\nabla\rho|^{2}+\rho(\tau+\mu-\lambda)V^{\tau+\mu-\lambda}\langle\nabla u,\nabla\rho\rangle+\rho V^{\tau+\mu-\lambda}\Delta\rho\\
=&V^{\tau+\lambda-\mu}|\nabla^{D}\rho|^{2}+\rho V^{\tau}\Delta^{D}\rho\\
\geq&(1+C)V^{\tau+\lambda-\mu}.
\endaligned\end{equation}
Multiplying \eqref{02Sec2thm9-2} by $\rho^{\alpha}|\phi|^{p}$ and integrating over $M$ yields
\begin{equation}\label{03Sec2thm9-2}\aligned
&(1+C)\int_{M}\rho^{\alpha}|\phi|^{p}V^{\tau+\lambda-\mu} ~dv_{g}\\
\leq&\int_{M}V^{\tau}\rm{div}^{D}(\rho\nabla^{D}\rho)\rho^{\alpha}|\phi|^{p} ~dv_{g} \\
=&\int_{M}\nabla(\rho V^{\tau+\mu-\lambda}\nabla\rho)\rho^{\alpha}|\phi|^{p} ~dv_{g} \\
=&-\int_{M}V^{\tau+\mu-\lambda}\langle\rho\nabla\rho,\nabla(\rho^{\alpha}|\phi|^{p})\rangle ~dv_{g}\\
=&-\int_{M}V^{\tau+\mu-\lambda}\langle\rho\nabla\rho,\alpha\rho^{\alpha-1}\nabla\rho|\phi|^{p}+p|\phi|^{p-2}\phi\rho^{\alpha}\nabla\phi
\rangle ~dv_{g}\\
=&-\alpha\int_{M}V^{\tau+\mu-\lambda}\rho^{\alpha}|\nabla\rho|^{2}|\phi|^{p} ~dv_{g}-p\int_{M}\rho^{\alpha+1}|\phi|^{p-2}\phi\langle\nabla\rho,\nabla\phi\rangle V^{\tau+\mu-\lambda} ~dv_{g}\\
=&-\alpha\int_{M}V^{\tau+\lambda-\mu}\rho^{\alpha}|\phi|^{p} ~dv_{g}-p\int_{M}\rho^{\alpha+1}|\phi|^{p-2}\phi\langle\nabla\rho,\nabla\phi\rangle V^{\tau+\mu-\lambda} ~dv_{g}.
\endaligned\end{equation}
Since $C+\alpha+1>0$, we can infer from \eqref{03Sec2thm9-2} that
$$\aligned
&(C+\alpha+1)\int_{M}\rho^{\alpha}|\phi|^{p}V^{\tau+\lambda-\mu} ~dv_{g}\leq
-p\int_{M}\rho^{\alpha+1}|\phi|^{p-2}\phi\langle\nabla\rho,\nabla\phi\rangle V^{\tau+\mu-\lambda} ~dv_{g},
\endaligned$$
which is equivalent to
\begin{equation}\label{04Sec2thm9-2}\aligned
(C+\alpha+1)\int_{M}\rho^{\alpha}|\phi|^{p}V^{\tau+\lambda-\mu} ~dv_{g}\leq
p\int_{M}V^{\tau+\mu-\lambda}\rho^{\alpha+1}|\phi|^{p-2}\phi\langle\nabla\rho,\nabla\phi\rangle ~dv_{g}.
\endaligned\end{equation}
By \eqref{04Sec2thm9-2}, we can get
\begin{equation}\label{01ad04Sec2thm9-2}\aligned
(C+\alpha+1)\int_{M}\rho^{\alpha}|\phi|^{p}V^{\tau+\lambda-\mu} ~dv_{g}
\leq p\int_{M}\rho^{\alpha+1}|\phi|^{p-1}\left|\langle\nabla\rho,\nabla\phi\rangle\right| V^{\tau+\mu-\lambda} ~dv_{g}.
\endaligned\end{equation}
By H\"{o}lder's and Young's inequalities, we have
\begin{equation}\label{05Sec2thm9-2}\aligned
&p\int_{M}\rho^{\alpha+1}|\phi|^{p-1}\left|\langle\nabla\rho,\nabla\phi\rangle\right| V^{\tau+\mu-\lambda} ~dv_{g}\\
\leq&p\left[\int_{M}\left(\rho^{\frac{\alpha(p-1)}{p}}|\phi|^{p-1}V^{\frac{p-1}{p}(\tau+\lambda-\mu)}\right)^{\frac{p}{p-1}} ~dv_{g}\right]^{\frac{p-1}{p}}
\left[\int_{M}\left(\rho^{\frac{\alpha+p}{p}}\left|\langle\nabla\rho,\nabla\phi\rangle\right|V^{\frac{1}{p}\tau+\left(\frac{1}{p}-2\right)(\lambda-\mu)}\right)^{p} ~dv_{g}\right]^{\frac{1}{p}}\\
=&p\left(\int_{M}\rho^{\alpha}|\phi|^{p}V^{\tau+\lambda-\mu} ~dv_{g}\right)^{\frac{p-1}{p}}\left(\int_{M}\rho^{\alpha+p}|\langle\nabla\rho,\nabla\phi\rangle|^{p}V^{\tau+\left(1-2p\right)(\lambda-\mu)} ~dv_{g}\right)^{\frac{1}{p}} \\
\leq&(p-1)\varepsilon^{-\frac{p}{p-1}}\int_{M}\rho^{\alpha}|\phi|^{p}V^{\tau+\lambda-\mu} ~dv_{g}
+\varepsilon^{p}\int_{M}\rho^{\alpha+p}|\langle\nabla\rho,\nabla\phi\rangle|^{p}V^{\tau+\left(1-2p\right)(\lambda-\mu)} ~dv_{g}. \endaligned\end{equation}
Thus, for any $\varepsilon>0$, from \eqref{01ad04Sec2thm9-2} and \eqref{05Sec2thm9-2}, we have
$$\aligned
&\int_{M}\rho^{\alpha+p}\left|\langle\nabla^{D}\rho,\nabla^{D}\phi\rangle\right|^{p}V^{\tau+(1-2p)(\lambda-\mu)} ~dv_{g}\\
\geq&\varepsilon^{-p}\left((C+\alpha+1)-(p-1)\varepsilon^{-\frac{p}{p-1}}\right)\int_{M}\rho^{\alpha}|\phi|^{p}V^{\tau+\lambda-\mu} ~dv_{g}.
\endaligned$$
Taking
$$\aligned
\varepsilon=\left(\frac{p}{C+\alpha+1}\right)^{\frac{p-1}{p}}
\endaligned$$
in the above inequality, we can get \eqref{01Sec2thm9-2} directly. This completes the proof of Theorem \ref{thm9}.
\end{proof}

We now prove, under an additional assumption on the weight functions $\rho(x)$ and $a(x)$, the following two-weight
Hardy-Poincar\'{e} inequality with a nonnegative remainder term. Our result recovers and improves the inequality
\eqref{1adaInt2}.

\begin{theorem}\label{thm5}
Let $(M^{n}, g,V=e^{u})$ be an n-dimensional Riemannian triple and
$\lambda,\mu\in \mathbb{R}$. Let $D=D^{\lambda,\mu}$ be the affine connection defined as in \eqref{01adsec02} and $\tau=(n+1)\lambda+\mu$. Let $\rho$ be a nonnegative functions on $M$ such that $|\nabla^{D}\rho|=1,~\Delta^{D}\rho\geq\frac{C}{\rho}V^{\lambda-\mu}$ and $\langle\nabla^{D}\rho,\nabla^{D} a\rangle\geq0$ in the sense of distribution, where $C>1$, $C+\alpha+1>0$, $\alpha\in\mathbb{R}$ and $a$ be a nonnegative functions on $M$. Then the following inequality
\begin{equation}\label{01Sec2thm05}\aligned
\int_{M}a\rho^{\alpha+p}|\langle\nabla^{D}\rho,\nabla^{D}\phi\rangle|^{p}V^{\tau+\lambda-\mu}~dv_{g}\geq & \left(\frac{C+\alpha+1}{p}\right)^{p}\int_{M}a\rho^{\alpha}|\phi|^{p}V^{\tau+\lambda-\mu}~dv_{g} \\
&+\left(\frac{C+\alpha+1}{p}\right)^{p-1}\int_{M}\rho^{\alpha+1}|\phi|^{p}\langle\nabla^{D}\rho,\nabla^{D}a\rangle V^{\tau+\lambda-\mu}~dv_{g}
\endaligned\end{equation}
holds for any $\phi$ in $C_{0}^{\infty}\left(M\setminus\rho^{-1}\{0\}\right)$, $1<p<+\infty.$
\end{theorem}
\begin{proof}
 By direct computation, together with assumptions $|\nabla^{D}\rho|=1$ and $\Delta^{D}\rho\geq\frac{C}{\rho}V^{\lambda-\mu}$, we have
\begin{equation}\label{02Sec2thm05}\aligned
&V^{\tau}D_{i}\left(\rho\nabla^{D}_{i}\rho\right)=\nabla_{i}\left(V^{\tau+\mu-\lambda}\rho\nabla_{i}\rho\right)\\
=&(\tau+\mu-\lambda)V^{\tau+\mu-\lambda}\rho\langle\nabla\rho,\nabla u\rangle+V^{\tau+\mu-\lambda}|\nabla\rho|^{2}+V^{\tau+\mu-\lambda}\rho\Delta\rho\\
=&V^{\tau}\rho\Delta^{D}\rho+V^{\tau+\lambda-\mu}|\nabla^{D}\rho|^{2}\\
\geq&(C+1)V^{\tau+\lambda-\mu}.
\endaligned\end{equation}
Multiplying both sides of \eqref{02Sec2thm05} by $a\rho^\alpha|\phi|^p$ and integrating over $M$ yields
$$\aligned
&(C+1)\int_{M}a\rho^{\alpha}|\phi|^{p}V^{\tau+\lambda-\mu}~dv_{g}\\
\leq&\int_{M}a\rho^{\alpha}|\phi|^{p}V^{\tau}D_{i}(\rho\nabla_{i}^{D}\rho)~dv_{g}\\
=&\int_{M}a\rho^{\alpha}|\phi|^{p}\nabla_{i}(\rho V^{\tau+\mu-\lambda}\nabla_{i}\rho)~dv_{g}\\
=&-\int_{M}V^{\tau+\mu-\lambda}\langle\rho\nabla\rho,\nabla\left(a\rho^{\alpha}|\phi|^{p}\right)\rangle~dv_{g}\\
=&-\int_{M}\rho^{\alpha+1}|\phi|^{p}V^{\tau+\mu-\lambda}\langle\nabla\rho,\nabla a\rangle~dv_{g}-\alpha\int_{M}a\rho^{\alpha}|\nabla\rho|^{2}|\phi|^{p}V^{\tau+\mu-\lambda}~dv_{g}\\
&-p\int_{M}a\rho^{\alpha+1}|\phi|^{p-2}\phi \langle\nabla\rho,\nabla\phi\rangle V^{\tau+\mu-\lambda}~dv_{g}.
\endaligned$$
After rearranging the terms, the above inequality can be written in the following form
\begin{equation}\label{03Sec2thm05}\aligned
A\leq&-p\int_{M}a\rho^{\alpha+1}|\phi|^{p-2}\phi \langle\nabla\rho,\nabla\phi\rangle V^{\tau+\mu-\lambda}~dv_{g}\\
\leq&p\int_{M}a\rho^{\alpha+1}|\phi|^{p-1}|\langle\nabla\rho,\nabla\phi\rangle| V^{\tau+\mu-\lambda}~dv_{g}.
\endaligned\end{equation}
where
$$\aligned
A=(C+1+\alpha)\int_{M}a\rho^{\alpha}|\phi|^{p}V^{\tau+\lambda-\mu}~dv_{g}+\int_{M}\rho^{\alpha+1}|\phi|^{p}\langle\nabla\rho,\nabla a\rangle V^{\tau+\mu-\lambda}~dv_{g}.
\endaligned$$
Applying  H\"{o}lder's inequality on the right hand side of \eqref{03Sec2thm05}, we get
\begin{equation}\label{004Sec2thm05}\aligned
A
\leq&p\int_{M}a^{\frac{p-1}{p}}\rho^{\frac{\alpha(p-1)}{p}}|\phi|^{p-1}a^{\frac{1}{p}}\rho^{\frac{\alpha}{p}+1} |\langle\nabla\rho,\nabla\phi\rangle| V^{\tau+\mu-\lambda}~dv_{g}\\
\leq&p\left[\int_{M}\left(a^{\frac{p-1}{p}}\rho^{\frac{\alpha(p-1)}{p}}|\phi|^{p-1}V^{\frac{(p-1)(\tau+\lambda-\mu)}{p}}\right)^{\frac{p}{p-1}}~dv_{g}\right]^{\frac{p-1}{p}}
\left[\int_{M}\left(a^{\frac{1}{p}}\rho^{\frac{\alpha}{p}+1} |\langle\nabla\rho,\nabla\phi\rangle| V^{\frac{\tau+(2p-1)(\mu-\lambda)}{p}}\right)^{p}~dv_{g}\right]^{\frac{1}{p}}\\
=&
p\left(\int_{M}a\rho^{\alpha}|\phi|^{p}V^{\tau+\lambda-\mu}~dv_{g}\right)^{\frac{p-1}{p}}
\left(\int_{M}a\rho^{\alpha+p}|\langle\nabla\rho,\nabla\phi\rangle|^{p} V^{\tau+(2p-1)(\mu-\lambda)}~dv_{g}\right)^{\frac{1}{p}}.
\endaligned\end{equation}
Using Young's inequality on the right hand side of \eqref{004Sec2thm05} again, we have
\begin{equation}\label{005Sec2thm05}\aligned
&\left(\int_{M}a\rho^{\alpha}|\phi|^{p}V^{\tau+\lambda-\mu}~dv_{g}\right)^{\frac{p-1}{p}}
\left(\int_{M}a\rho^{\alpha+p}|\langle\nabla\rho,\nabla\phi\rangle|^{p} V^{\tau+(2p-1)(\mu-\lambda)}~dv_{g}\right)^{\frac{1}{p}}\\
\leq&\frac{p-1}{p}\left(\frac{1}{\varepsilon^{\frac{p}{p-1}}}\int_{M}a\rho^{\alpha}|\phi|^{p}V^{\tau+\lambda-\mu}~dv_{g}
\right)+\frac{1}{p}\left(\varepsilon^{p}\int_{M}a\rho^{\alpha+p}|\langle\nabla\rho,\nabla\phi\rangle|^{p} V^{\tau+(2p-1)(\mu-\lambda)}~dv_{g}\right)
\endaligned\end{equation}
for any $\varepsilon>0$.
Substituting \eqref{005Sec2thm05} into \eqref{004Sec2thm05}, we can get
\begin{equation}\label{04Sec2thm05}\aligned
A
\leq(p-1)\varepsilon^{-\frac{p}{p-1}}\int_{M}a\rho^{\alpha}|\phi|^{p}V^{\tau+\lambda-\mu}~dv_{g}+\varepsilon^{p}\int_{M}a\rho^{\alpha+p}|\langle\nabla\rho,\nabla\phi\rangle|^{p} V^{\tau+(2p-1)(\mu-\lambda)}~dv_{g}.
\endaligned\end{equation}
Reorganizing yields
$$\aligned
&\left(C+1+\alpha-(p-1)\varepsilon^{-\frac{p}{p-1}}\right)\int_{M}a\rho^{\alpha}|\phi|^{p}V^{\tau+\lambda-\mu}~dv_{g}+\int_{M}\rho^{\alpha+1}|\phi|^{p}\langle\nabla\rho,\nabla a\rangle V^{\tau+\mu-\lambda}~dv_{g}\\
\leq&\varepsilon^{p}\int_{M}a\rho^{\alpha+p}|\langle\nabla\rho,\nabla\phi\rangle|^{p} V^{\tau+(2p-1)(\mu-\lambda)}~dv_{g}.
\endaligned$$
Let $f(\varepsilon)=\varepsilon^{-p}\left[C+1+\alpha-(p-1)\varepsilon^{-\frac{p}{p-1}}\right]$, we have
$$\aligned
&f(\varepsilon)\int_{M}a\rho^{\alpha}|\phi|^{p}V^{\tau+\lambda-\mu}~dv_{g}+\varepsilon^{-p}\int_{M}\rho^{\alpha+1}|\phi|^{p}\langle\nabla\rho,\nabla a\rangle V^{\tau+\mu-\lambda}~dv_{g}\\
\leq&\int_{M}a\rho^{\alpha+p}|\langle\nabla\rho,\nabla\phi\rangle|^{p} V^{\tau+(2p-1)(\gamma-\alpha)}~dv_{g}\\
=&\int_{M}a\rho^{\alpha+p}|\langle\nabla^{D}\rho,\nabla^{D}\phi\rangle|^{p} V^{\tau+\lambda-\mu}~dv_{g}.
\endaligned$$
Note that the function $f$ attains the maximum for $\varepsilon_0=$ $\left(\frac p{C+\alpha+1}\right)^{\frac{p-1}p}$ and this maximum value is equal to $\left(\frac{C+\alpha+1}{p}\right)^p$. Therefore we obtain the desired inequality
\begin{equation}\label{05Sec2thm05}\aligned
\int_{M}a\rho^{\alpha+p}|\langle\nabla^{D}\rho,\nabla^{D}\phi\rangle|^{p} V^{\tau+\lambda-\mu}~dv_{g} \geq&\left(\frac{C+\alpha+1}{p}\right)^{p}\int_{M}a\rho^{\alpha}|\phi|^{p}V^{\tau+\lambda-\mu}~dv_{g}  \\
&+\left(\frac{C+\alpha+1}{p}\right)^{p-1}\int_{M}\rho^{\alpha+1}|\phi|^{p}\langle\nabla^{D}\rho,\nabla^{D} a \rangle V^{\tau+\lambda-\mu}~dv_{g}.
\endaligned\end{equation}
The proof of Theorem \ref{thm5} is completed.
\end{proof}

The following Heisenberg-Pauli-Weyl type inequalities with respect to $D^{\lambda,\mu}$ has been
proved:

\begin{theorem}\label{thm10}
Let $(M^{n}, g,V=e^{u})$ be an n-dimensional Riemannian triple and
$\lambda,\mu\in \mathbb{R}$. Let $D=D^{\lambda,\mu}$ be the affine connection defined as in \eqref{01adsec02} and $\tau=(n+1)\lambda+\mu$.
Let $\rho$ be a nonnegative function on $M$ such that $|\nabla^{D}\rho|=1$ and $\Delta^{D}\rho\geq \frac{C}{\rho}V^{\lambda-\mu}$ in the sense of distributions, where $C>0$ is a constant.
Then for any compactly supported smooth function $\phi\in C^{\infty}_{0} (M\backslash\rho^{-1}\{0\})$, we have
\begin{equation}\label{01Sec2thm10}\aligned
\left(\int_{M}\rho^{2}\phi^{2}V^{\tau+\lambda-\mu} ~dv_{g}\right)\left(\int_{M}|\nabla^{D}\phi|^{2}V^{\tau+\lambda-\mu} ~dv_{g}\right)
\geq\frac{(C+1)^{2}}{4}\left(\int_{M}\phi^{2}V^{\tau+\lambda-\mu} ~dv_{g}\right)^{2}.
\endaligned\end{equation}
Moreover, for $C>1$, we have
\begin{equation}\label{02Sec2thm10}\aligned
\left(\int_{M}\rho^{4}\phi^{2}V^{\tau+\lambda-\mu} ~dv_{g}\right)\left(\int_{M}(\Delta^{D}\phi)^{2}V^{\tau+\mu-\lambda} ~dv_{g}\right)
\geq\frac{(C+1)^{4}}{16}\left(\int_{M}\phi^{2}V^{\tau+\lambda-\mu} ~dv_{g}\right)^{2}.
\endaligned\end{equation}
\end{theorem}
\begin{proof}
It follows from $|\nabla^{D}\rho|=1$ and $\Delta^{D}\rho\leq \frac{C}{\rho}V^{\lambda-\mu}$ that
\begin{equation}\label{03Sec2thm10}\aligned
\Delta^{D}\rho^{2}=&V^{\mu-\lambda}\left(\Delta\rho^{2}+(2\mu+n\lambda)\langle\nabla u,\nabla \rho^{2}\rangle\right)\\
=&V^{\mu-\lambda}\left(2|\nabla\rho|^{2}+2\rho\Delta \rho+(2\mu+n\lambda)2\rho\langle\nabla u,\nabla \rho\rangle\right)\\
=&2|\nabla^{D}\rho|^{2}V^{\lambda-\mu}+2\rho\Delta^{D}\rho\\
\geq&(2C+2)V^{\lambda-\mu}.
\endaligned\end{equation}
Multiplying both sides of \eqref{03Sec2thm10} by $\phi^{2}V^{\tau}$ and integrating over $M$ yields
\begin{equation}\label{04Sec2thm10}\aligned
\int_{M}\phi^{2}V^{\tau}\Delta^{D}\rho^{2} ~dv_{g}\geq(2C+2)\int_{M}\phi^{2}V^{\tau+\lambda-\mu} ~dv_{g}.
\endaligned\end{equation}
On the other hand, applying the integrating by parts and H\"{o}lder's inequality, we have
\begin{equation}\label{05Sec2thm10}\aligned
&\int_{M}\phi^{2}V^{\tau}\Delta^{D}\rho^{2} ~dv_{g}\\
=&-\int_{M}\langle\nabla\rho^{2},\nabla\phi^{2}\rangle V^{\tau+\mu-\lambda} ~dv_{g}\\
=&-\int_{M}4\rho\phi\langle\nabla\rho,\nabla\phi\rangle V^{\tau+\mu-\lambda} ~dv_{g}\\
\leq&4\int_{M}\rho|\phi||\nabla\rho||\nabla\phi| V^{\tau+\mu-\lambda} ~dv_{g}\\
\leq&4\left(\int_{M}|\rho\phi\nabla\rho|^{2}V^{\tau+\mu-\lambda} ~dv_{g}\right)^{\frac{1}{2}}\left(\int_{M}|\nabla\phi|^{2}V^{\tau+\mu-\lambda} ~dv_{g}\right)^{\frac{1}{2}}\\
=&4\left(\int_{M}\rho^{2}\phi^{2}|\nabla\rho|^{2}V^{\tau+\mu-\lambda} ~dv_{g}\right)^{\frac{1}{2}}\left(\int_{M}|\nabla\phi|^{2}V^{\tau+\mu-\lambda} ~dv_{g}\right)^{\frac{1}{2}}\\
=&4\left(\int_{M}\rho^{2}\phi^{2}V^{\tau+\lambda-\mu} ~dv_{g}\right)^{\frac{1}{2}}\left(\int_{M}|\nabla^{D}\phi|^{2}V^{\tau+\lambda-\mu} ~dv_{g}\right)^{\frac{1}{2}}.
\endaligned\end{equation}
Substituting \eqref{05Sec2thm10} into \eqref{04Sec2thm10}, we can obtain \eqref{01Sec2thm10} directly.

From \eqref{04Sec2thm10}, we have
\begin{equation}\label{01ad06Sec2thm10}\aligned
\left|\frac{1}{4}\int_{M}\phi^{2} V^{\tau}\Delta^{D}\rho^{2} ~dv_{g}\right|
\geq\frac{C+1}{2}\int_{M}\phi^{2}V^{\tau+\lambda-\mu} ~dv_{g}.
\endaligned\end{equation}
Using the integration by parts again, we can get
\begin{equation}\label{02ad06Sec2thm10}\aligned
&\left|\frac{1}{4}\int_{M}\phi^{2} V^{\tau}\Delta^{D}\rho^{2} ~dv_{g}\right|\\
=&\left|-\frac{1}{4}\int_{M}\langle\nabla^{D}\rho^{2},\nabla\phi^{2}\rangle V^{\tau} ~dv_{g}\right|\\
=&\left|-\frac{1}{4}\int_{M}\langle\nabla\rho^{2},\nabla\phi^{2}\rangle V^{\tau+\mu-\lambda}~dv_{g}\right|\\
=&\left|-\frac{1}{4}\int_{M}4\rho\phi\langle\nabla\rho,\nabla\phi\rangle V^{\tau+\mu-\lambda}~dv_{g}\right|\\
=&\int_{M}\rho|\phi||\nabla\rho||\nabla\phi| V^{\tau+\mu-\lambda}~dv_{g}\\
\leq&\left(\int_{M}\rho^{4}\phi^{2}|\nabla\rho|^{2}V^{\tau+\mu-\lambda} ~dv_{g}\right)^{\frac{1}{2}}\left(\int_{M}\frac{|\nabla\phi|^{2}}{\rho^{2}}V^{\tau+\mu-\lambda} ~dv_{g}\right)^{\frac{1}{2}},
\endaligned\end{equation}
where, we use H\"{o}lder's inequality in the last equation.
On the other hand, letting $\alpha=0$ in \eqref{02Sec3thm7-2}, we have
\begin{equation}\label{07Sec2thm10}\aligned
\int_{M}(\Delta^{D}\phi)^{2}V^{\tau+\mu-\lambda} ~dv_{g}
\geq\frac{(C+1)^{2}}{4}\int_{M}\frac{|\nabla^{D}\phi|^{2}}{\rho^{2}}V^{\tau+\lambda-\mu} ~dv_{g}.
\endaligned\end{equation}
Inserting \eqref{01ad06Sec2thm10} and \eqref{07Sec2thm10} into \eqref{02ad06Sec2thm10}, the inequality \eqref{02Sec2thm10} can be obtained naturally. This completes the proof of
Theorem \ref{thm10}.
\end{proof}

We now prove a weight $L^{p}$ form of the Hardy type inequality with respect to $D^{\lambda,\mu}$, where our new inequality contains more
remainder terms (or the improved Hardy type inequality).

\begin{theorem}\label{thm11}
Let $(M^{n}, g,V=e^{u})$ be an n-dimensional Riemannian triple and
$\lambda,\mu\in \mathbb{R}$. Let $D=D^{\lambda,\mu}$ be the affine connection defined as in \eqref{01adsec02} and $\tau=(n+1)\lambda+\mu$.
Let $\delta$ be a positive function, and $\rho$ and a be nonnegative functions on $M$ such that $|\nabla^{D}\rho|=1$, $\Delta^{D}\rho\geq\frac{C}{\rho}V^{\lambda-\mu}, \langle\nabla^{D} a, \nabla^{D}\rho\rangle\geq0$ and $-V^{\tau}\rm{div}^{D}\left(a\rho^{1-C}\nabla^{D}\delta\right)\geq0$ in the sense of distribution, where $C>1$ is a constant, $C+\alpha-1>0$ and $\alpha\in\mathbb{R}$. Then the following inequality
\begin{equation}\label{01Sec2thm11}\aligned
\int_{M}a\rho^{\alpha}|\nabla^{D}\phi|^{2}V^{\tau+\lambda-\mu} ~dv_{g} \geq&\left(\frac{C+\alpha-1}{2}\right)^{2}\int_{M}a\rho^{\alpha-2}\:\phi^{2}V^{\tau+\lambda-\mu} ~dv_{g} \\
&+\frac{C+\alpha-1}{2}\int_{M}\rho^{\alpha-1}\:\phi^{2} \langle\nabla^{D}a, \nabla^{D}\rho\rangle  V^{\tau+\lambda-\mu} ~dv_{g} \\
&+\frac{1}{4}\int_{M}a\rho^{\alpha}\frac{|\nabla^{D}\delta|^{2}}{\delta^{2}}\phi^{2}V^{\tau+\lambda-\mu} ~dv_{g}
\endaligned\end{equation}
holds for any $\phi\in C_0^{\infty}\big(M\setminus\rho^{-1}\{0\}\big)$.
\end{theorem}
\begin{proof}
Let $\phi=\rho^\beta\psi$, where $\psi\in C_0^\infty\big(M\setminus\rho^{-1}\{0\}\big)$ and $\beta<0$. Then we have
$$\aligned
\left|\nabla\phi\right|^{2}
=&\left|\beta\rho^{\beta-1}\psi\nabla\rho+\rho^\beta\nabla\psi\right|^{2}\\
=&\beta^{2}\rho^{2(\beta-1)}\psi^{2}|\nabla\rho|^{2}+\rho^{2\beta}|\nabla\psi|^{2}+2\beta\psi\:\rho^{2\beta-1}\langle\nabla\rho, \nabla\psi\rangle.
\endaligned$$
Multiplying the above inequality by $V^{\tau+\mu-\lambda}$ and using $|\nabla^{D}\rho|=1$, we obtain
\begin{equation}\label{02Sec2thm11}\aligned
V^{\tau+\mu-\lambda}\left|\nabla\phi\right|^{2}
=&\beta^{2}\rho^{2(\beta-1)}\psi^{2}V^{\tau+\lambda-\mu}+\rho^{2\beta}|\nabla\psi|^{2}V^{\tau+\mu-\lambda}\\
&+\beta\rho^{2\beta-1}\langle\nabla\rho, \nabla\psi^{2}\rangle V^{\tau+\mu-\lambda}.
\endaligned\end{equation}
Multiplying both sides of \eqref{02Sec2thm11} by $a(x)\rho^{\alpha}$ and applying the integration by parts over $M$ gives
\begin{equation}\label{03Sec2thm11}\aligned
&\int_{M}a\rho^{\alpha}|\nabla\phi|^{2}V^{\tau+\mu-\lambda} ~dv_{g}\\
=&\int_{M}a\beta^{2}\rho^{\alpha+2(\beta-1)}\psi^{2}V^{\tau+\lambda-\mu} ~dv_{g}
+\int_{M}a\rho^{\alpha+2\beta}|\nabla\psi|^{2}V^{\tau+\mu-\lambda} ~dv_{g}\\
&+\int_{M}a\beta\rho^{\alpha+2\beta-1}\langle\nabla\rho, \nabla\psi^{2}\rangle V^{\tau+\mu-\lambda} ~dv_{g}\\
=&\beta^{2}\int_{M}a\rho^{\alpha+2(\beta-1)}\:\psi^{2}V^{\tau+\lambda-\mu} ~dv_{g}
-\beta\int_{M}\mathrm{div}\left(a\rho^{\alpha+2\beta-1}\:V^{\tau+\mu-\lambda}\nabla\rho\right)\psi^{2} ~dv_{g}\\
&+\int_{M}a\rho^{\alpha+2\beta}|\nabla\psi|^{2}V^{\tau+\mu-\lambda} ~dv_{g}.
\endaligned\end{equation}
Since $|\nabla^{D}\rho|=1$ and $\Delta^{D}\rho\geq\frac{C}{\rho}V^{\lambda-\mu}$, we obtain the following inequality
\begin{equation}\label{04Sec2thm11}\aligned
&\mathrm{div}\left(a\rho^{\alpha+2\beta-1}\:V^{\tau+\mu-\lambda}\nabla\rho\right)\\
=&\rho^{\alpha+2\beta-1}V^{\tau+\mu-\lambda}\langle\nabla\rho, \nabla a\rangle+(\alpha+2\beta-1)a\rho^{\alpha+2\beta-2}|\nabla\rho|^{2}V^{\tau+\mu-\lambda}\\
&+a\rho^{\alpha+2\beta-1}V^{\tau+\mu-\lambda}\Delta\rho
+a(\tau+\mu-\lambda)\rho^{\alpha+2\beta-1}V^{\tau+\mu-\lambda}\langle\nabla\rho, \nabla u\rangle\\
=&\rho^{\alpha+2\beta-1}V^{\tau+\mu-\lambda}\langle\nabla\rho, \nabla a\rangle+(\alpha+2\beta-1)a\rho^{\alpha+2\beta-2}|\nabla^{D}\rho|^{2}V^{\tau+\lambda-\mu}\\
&+a\rho^{\alpha+2\beta-1}V^{\tau}\Delta^{D}\rho\\
\geq&(C+\alpha+2\beta-1)a\rho^{\alpha+2\beta-2}V^{\tau+\lambda-\mu}+\rho^{\alpha+2\beta-1}V^{\tau+\mu-\lambda}\langle\nabla\rho, \nabla a\rangle.
\endaligned\end{equation}
Multiplying both sides of \eqref{04Sec2thm11} by $\beta\psi^{2}$ and integrating over $M$ yields
\begin{equation}\label{05Sec2thm11}\aligned
-\beta\int_{M}\mathrm{div}\left(a\rho^{\alpha+2\beta-1}\:V^{\tau+\mu-\lambda}\nabla\rho\right)\psi^{2} ~dv_{g}
\geq&-\beta\left(C+\alpha+2\beta-1\right)\int_{M}a\rho^{\alpha+2\beta-2}\psi^{2}V^{\tau+\lambda-\mu} ~dv_{g} \\
&-\beta\int_{M}\rho^{\alpha+2\beta-1}\langle\nabla a,\nabla\rho\rangle V^{\tau+\mu-\lambda}\psi^{2} ~dv_{g}.
\endaligned\end{equation}
Hence combining the inequalities \eqref{03Sec2thm11} and \eqref{05Sec2thm11}, we obtain
\begin{equation}\label{06Sec2thm11}\aligned
&\int_{M}a\rho^{\alpha}|\nabla\phi|^{2}V^{\tau+\mu-\lambda} ~dv_{g}\\
\geq&\beta^{2}\int_{M}a\rho^{\alpha+2\beta-2}\psi^{2}V^{\tau+\lambda-\mu} ~dv_{g}-\beta(C+\alpha+2\beta-1)\int_{M}a\rho^{\alpha+2\beta -2}\psi^{2}V^{\tau+\lambda-\mu} ~dv_{g}\\
&-\beta\int_{M}\rho^{\alpha+2\beta-1}\langle\nabla a,\nabla\rho\rangle \psi^{2}V^{\tau+\mu-\lambda} ~dv_{g}
+\int_{M}a\rho^{\alpha+2\beta}|\nabla\psi|^{2}V^{\tau+\mu-\lambda} ~dv_{g}\\
=&f\left(\beta\right)\int_{M}a\rho^{\alpha+2\beta -2}\psi^{2}V^{\tau+\lambda-\mu} ~dv_{g}-\beta\int_{M}\rho^{\alpha+2\beta -1}\langle\nabla a,\nabla\rho\rangle \psi^{2}V^{\tau+\mu-\lambda} ~dv_{g}\\
&+\int_{M}a\rho^{\alpha+2\beta}|\nabla\psi|^{2}V^{\tau+\mu-\lambda} ~dv_{g},
\endaligned\end{equation}
where $f\left(\beta\right)=\beta^{2}-\beta\left(C+\alpha+2\beta-1\right).$
Since $f\left(\beta\right)$ attains the maximum for $\beta_{0}=\frac{1-\alpha-C}{2}$ and this maximum value is equal to $f\left(\beta_{0}\right)=\left(\frac{C+\alpha-1}{2}\right)^{2}$,
the inequality \eqref{06Sec2thm11} becomes
\begin{equation}\label{07Sec2thm11}\aligned
&\int_{M}a\rho^{\alpha}|\nabla\phi|^{2}V^{\tau+\mu-\lambda} ~dv_{g}\\ \geq&\left(\frac{C+\alpha-1}{2}\right)^{2}\int_{M}a\rho^{-C-1}\psi^{2}V^{\tau+\lambda-\mu} ~dv_{g}
+\int_{M}a\rho^{1-C}|\nabla\psi|^{2}V^{\tau+\mu-\lambda} ~dv_{g}\\
&+\frac{C+\alpha-1}{2}\int_{M}\rho^{-C}\psi^{2}\langle\nabla a,\nabla\rho\rangle V^{\tau+\mu-\lambda} ~dv_{g}.
\endaligned\end{equation}
We now focus on the second term on the right hand side of \eqref{07Sec2thm11}. Let us define a new variable $\varphi:=\delta^{-\frac{1}{2}}\psi$,
 where $\delta\in C^{\infty}_{0}(M\setminus\rho^{-1}\{0\})$ is a positive function. It is clear that
$$\aligned
|\nabla\psi|^{2}
=&\left|\frac{1}{2}\delta^{-\frac{1}{2}}\varphi\nabla\delta+\delta^{\frac{1}{2}}\nabla\varphi\right|^{2}\\
=&\frac{1}{4}\delta^{-1}\:|\nabla\delta|^{2}\varphi^{2}+\delta|\nabla\varphi|^{2}+\langle\nabla\delta,\nabla\varphi\rangle\varphi\\
\geq&\frac{1}{4}\delta^{-1}\:|\nabla\delta|^{2}\varphi^{2}+\frac{1}{2}\langle\nabla\delta,\nabla\varphi^{2}\rangle.
\endaligned$$
Therefore
\begin{equation}\label{08Sec2thm11}\aligned
\int_{M}a\rho^{1-C}|\nabla\psi|^{2}V^{\tau+\mu-\lambda} ~dv_{g} \geq&\frac{1}{4}\int_{M}a\rho^{1-C}\varphi^{2}\:\delta^{-1}\:|\nabla\delta|^{2}V^{\tau+\mu-\lambda} ~dv_{g} \\
&+\frac{1}{2}\int_{M}a\rho^{1-C}\langle\nabla\delta,\nabla\varphi^{2}\rangle V^{\tau+\mu-\lambda} ~dv_{g}.
\endaligned\end{equation}
Here, first applying the integration by parts to the second term on the right hand side of \eqref{08Sec2thm11} and then using the differential inequality $-V^{\tau}\rm{div}^{D}\left(a\rho^{1-C}\nabla^{D}\delta\right)\geq0$ , we get
\begin{equation}\label{09Sec2thm11}\aligned
&\int_{M}a\rho^{1-C}\langle\nabla\delta,\nabla\varphi^{2}\rangle V^{\tau+\mu-\lambda} ~dv_{g}\\
=&\int_{M} a\rho^{1-C}\langle\nabla^{D}\delta,\nabla\varphi^{2}\rangle V^{\tau} ~dv_{g}\\
=&-\int_{M}\varphi^{2}\rm{div}\left(V^{\tau}a\rho^{1-C}\nabla^{D}\delta\right) ~dv_{g}\\
=&-\int_{M}\varphi^{2}D_{i}\left(a\rho^{1-C}\nabla^{D}_{i}\delta\right)V^{\tau} ~dv_{g}\\
\geq&0.
\endaligned\end{equation}
Then combining \eqref{08Sec2thm11} and \eqref{09Sec2thm11}, we have
\begin{equation}\label{10Sec2thm11}\aligned
\int_{M}a\rho^{1-C}|\nabla\psi|^{2}V^{\tau+\mu-\lambda} ~dv_{g}\geq&\frac{1}{4}\int_{M}a\rho^{1-C}\varphi^{2}\:\delta^{-1}\:|\nabla\delta|^{2}V^{\tau+\mu-\lambda} ~dv_{g}.
\endaligned\end{equation}
Taking $\varphi=\delta^{-\frac{1}{2}}\rho^{\frac{C+\alpha-1}{2}}\phi$ into \eqref{10Sec2thm11}, then \eqref{10Sec2thm11} becomes
\begin{equation}\label{11Sec2thm11}\aligned
\int_{M}a\rho^{1-C}|\nabla\psi|^{2}V^{\tau+\mu-\lambda} ~dv_{g}\geq\frac{1}{4}\int_{M}a\rho^{\alpha}\frac{|\nabla\delta|^{2}}{\delta^{2}}\:\phi^{2}V^{\tau+\mu-\lambda} ~dv_{g}.
\endaligned\end{equation}
Finally, by using the equality $\psi=\rho^{-\beta}\phi=\rho^{\frac{C+\alpha-1}{2}}\phi$ and taking \eqref{11Sec2thm11} into \eqref{07Sec2thm11}, we get the desired inequality \eqref{01Sec2thm11}.
This completes the proof of
Theorem \ref{thm11}.
\end{proof}

\vskip 6mm
\noindent{\bf Acknowledgements}

\noindent
The research of authors is supported by NSFC (No.12101530), the Science and Technology
Project of Henan Province (No.232102310321), and the Key Scientific Research Program in
Universities of Henan Province (Nos.21A110021, 22A110021) and Nanhu Scholars Program
for Young Scholars of XYNU (No.2023).

\end{document}